\documentclass[a4paper,12pt,reqno]{amsart}
\usepackage{amssymb}
\usepackage{amsmath}
\usepackage{anyfontsize}
\usepackage{hyperref}
\usepackage{mathtools}
\usepackage{mathrsfs}
\usepackage{bbm}
\usepackage{tikz}
\usetikzlibrary{commutative-diagrams}
\usepackage{tikz-cd}
\usepackage{bm}
\usepackage{bm}
\usepackage{geometry}
\usepackage{colordvi}
\usepackage{color}

\DeclareMathOperator{\m}{\mathfrak{m}}

\DeclareMathOperator{\ovep}{\overline{\varepsilon}}

\allowdisplaybreaks

\numberwithin{equation}{section}

\theoremstyle{definition}
\newtheorem{theorem}{Theorem}[section]
\newtheorem{corollary}[theorem]{Corollary}
\newtheorem{lemma}[theorem]{Lemma}
\newtheorem{definition}[theorem]{Definition}
\newtheorem{proposition}[theorem]{Proposition}
\newtheorem{example}[theorem]{Example}
\newtheorem{remark}[theorem]{Remark}
\newtheorem{que}[theorem]{Question}

\begin{document}


\title{Chevalley property of module-finite Hopf algebras and discriminant ideals}

\author{Yimin Huang}
\address{School of Mathematical Sciences, Fudan University, Shanghai 200433, China}
\email{21110180008@m.fudan.edu.cn}

\author{Tiancheng Qi}
\address{School of Mathematical Sciences, Fudan University, Shanghai 200433, China}
\email{tcqi21@m.fudan.edu.cn}

\author{Quanshui Wu}
\address{School of Mathematical Sciences, Fudan University, Shanghai 200433, China}
\email{qswu@fudan.edu.cn}

\author{Ruipeng Zhu}
\address{School of Mathematics, Shanghai University of Finance and Economics, Shanghai 200433, China}
\email{zhuruipeng@sufe.edu.cn}



\begin{abstract}
In this paper, we study the Chevalley property of Cayley-Hamilton Hopf algebras in the sense of De Concini-Procesi-Reshetikhin-Rosso using discriminant ideals.
For any affine Cayley-Hamilton Hopf algebra $(H,C,\text{tr})$ whose identity fiber algebra has the Chevalley property, we prove that an irreducible $H$-module $V$ has the property that $V\otimes W$ is a completely reducible $H$-module for every irreducible $H$-module $W$ if and only if $V$ is annihilated by the lowest discriminant ideal of $(H,C,\text{tr})$, which establishes a bridge between the tensor-nondegenerate behaviour of the irreducible representations of $H$ and the lowest discriminant ideal of $(H,C,\text{tr})$. Using discriminant ideals, 
we prove that an affine Cayley-Hamilton Hopf algebra $(H,C,\text{tr})$ has the Chevalley property
if and only if its identity fiber algebra $H/\mathfrak{m}_{\ovep}H$ has the Chevalley property and all the discriminant ideals of $(H,C,\text{tr})$ are trivial, thereby resolving a question posed by Huang-Mi-Qi-Wu. Moreover, it is shown that the lowest discriminant subvariety $\mathcal{V}_{\ell}$ of the algebraic group $\operatorname{maxSpec}C$ is a closed subgroup, which reflects the rigid nature of $\mathcal{V}_{\ell}$ and is effective in determining the lowest discriminant subvarieties in certain examples of low GK dimension. Consequently, the lowest discriminant subvariety $\mathcal{V}_{\ell}$ is a smooth equidimensional algebraic variety, a property not generally shared by other discriminant subvarieties. This rigidity property provides a method, via the lowest discriminant ideals, for constructing a large family of Hopf algebras with the Chevalley property and finite GK dimension. The results are illustrated through applications to the big quantized Borel subalgebras at roots of unity and to certain Artin-Schelter Gorenstein Hopf algebras of low GK dimension. In particular, the framework yields (non-finite) tensor categories with the Chevalley property arising from some big quantum groups at roots of unity.

\end{abstract}

\thanks{Q.-S. Wu is supported by the NSFC (Grant No. 12471032),  R.-P. Zhu  is supported by the NSFC (Grant No. 12301052).}

\maketitle	

\section*{Introduction}

Over the past three decades, Hopf algebras of finite Gelfand-Kirillov (abbreviated GK) dimension  have attracted considerable attention. See for example \cite{MR4993170,MR4600057,MR4298502,MR4201485,MR1482982,MR1676211,MR2661247,MR2732991,MR2178656,mi2025lowest,MR4042591,MR3490761} and the references therein. Among these, many important examples are \textit{affine} (i.e., finitely generated as algebras) and admit large central Hopf subalgebras (i.e., central Hopf subalgebras over which are finitely generated modules), including many quantum groups at roots of unity \cite{MR1103601,MR4600057,MR1296515,MR1124981,MR1351503,MR1288995,MR2178656} and many Artin-Schelter Gorenstein Hopf algebras of low GK dimension \cite{MR2320655,MR2661247,MR3490761,MR4042591,MR4281365}. Understanding the representation theory of such Hopf algebras is a fundamental and natural problem.

In \cite{HMQWChev2025}, the authors proved that \textit{all affine Hopf algebras admitting large central Hopf subalgebras} can be endowed with a Cayley-Hamilton Hopf algebra structure (see Definition \ref{def-CHHopf}) in the sense of De Concini-Procesi-Reshetikhin-Rosso \cite{MR2178656} via the Hattori-Stallings trace map (see Example \ref{eg-HStrace}). Hence a large class of important Hopf algebras of finite GK dimension fall within the framework of Cayley-Hamilton Hopf algebras. The theory and methods developed for Cayley-Hamilton Hopf algebras \cite{MR883869,MR1288995,MR2178656,MR3886192,mi2025lowest,HMQWChev2025} enable the representation theory and ring-theoretic properties of this broad class of Hopf algebras to be studied in a unified manner, without requiring case-by-case analysis. On the other hand, if $(H,C,\text{tr})$ is an affine Cayley-Hamilton Hopf algebra, by \cite[Theorem 4.5 (a)]{MR1288995}, $H$ is a finitely generated $C$-module and $C$ is an affine $\mathbbm{k}$-algebra.  Another advantage of working within the axiomatic framework of affine Cayley-Hamilton Hopf algebras is that, in many important cases of affine module-finite Hopf algebras, \textit{one can choose an appropriate trace map flexibly} (not only the Hattori-Stallings trace) to study the representation theory of these Hopf algebras.






 The \textit{Chevalley property}, introduced in \cite{MR1852304}, is an active research topic in representation theory and tensor categories, see \cite{MR1852304,MR1839475,MR3035999,MR3242743,MR3677868,MR4276314,MR4228433,MR4801656} for a partial list. A (rigid) tensor category is said to have the Chevalley property if the tensor product of any two irreducible/simple objects is completely reducible/semisimple. A Hopf algebra $H$ (with bijective antipode) is said to have the  Chevalley property if the tensor category of finite-dimensional $H$-modules has the Chevalley property (the dual Chevalley property can be defined similarly by considering the category of finite-dimensional comodules). This terminology is named after a classical result of Chevalley, which states that for any group, the tensor product of any two finite-dimensional irreducible representations over $\mathbb{C}$ is completely reducible.  In fact, any group algebra and the universal enveloping algebra of any Lie algebra over an algebraically closed field of characteristic zero have the Chevalley property \cite[Theorem 4.12.1]{MR3242743}. All finite-dimensional semisimple or basic Hopf algebras have the Chevalley property. For example, the small quantum Borel subgroups $\mathfrak{u}_{\epsilon}^{\geq 0}(\mathfrak{g})$ at roots of unity are basic \cite[Theorem 2.3 (a)]{MR1863398}. Molnar proved that the group algebra of a finite group $G$ over a field of characteristic $p$ has the Chevalley property if and only if $G$ has a unique Sylow $p$-subgroup \cite[Corollary 1]{MR639443}. For any restricted Lie algebra $(\mathfrak{g},[p])$, the restricted enveloping algebra $\mathfrak{u}(\mathfrak{g})$ has the Chevalley property if and only if the Jacobson radical of $\mathfrak{u}(\mathfrak{g})$ is generated by the $p$-nilpotent radical of $\mathfrak{g}$ \cite[Proposition 2]{MR639443}. Etingof-Gelaki \cite{MR1996801} showed that all finite-dimensional triangular Hopf algebras over an algebraically closed field of characteristic zero have the Chevalley property. Moreover, Etingof-Gelaki proved that any finite symmetric tensor category over an algebraically closed field of characteristic zero is equivalent to the category of finite-dimensional representations of some finite-dimensional triangular Hopf algebra \cite[Theorem 5.2]{MR1996801}; consequently, all such tensor categories have the Chevalley property. All finite-dimensional triangular Hopf algebras with the Chevalley property over an algebraically closed field of characteristic $\neq 2$ have been completely classified \cite{MR1852304,MR1839475,MR4276314,MR4228433}. There are finite-dimensional quasitriangular Hopf algebras without the Chevalley property (e.g., the Lusztig's small quantum groups $\mathfrak{u}_{\epsilon}(\mathfrak{sl}_{2})$ at roots of unity $\epsilon$). Although extensive work has been done on finite-dimensional Hopf algebras with the Chevalley property, little is known in the infinite-dimensional case. 

 

 In the rest of the introduction, we assume that $(H,C,\text{tr})$ is an affine Cayley-Hamilton Hopf algebra over an algebraically closed field $\mathbbm{k}$ of characteristic zero. Then $\operatorname{maxSpec}C$ (the maximal spectrum of $C$) is an affine algebraic group with respect to convolution; and for any $\mathfrak{m}\in \operatorname{maxSpec}C$, the \textit{fiber algebra} $H/\mathfrak{m}H$ at $\mathfrak{m}$ is a finite-dimensional algebra over $\mathbbm{k}$. It is well-known that the set of irreducible representations of $H$ is the disjoint union of the irreducible representations over its fiber algebras (see \S\ref{sec-modufinalg}). Let $\varepsilon$ be the counit of $H$, and $\ovep\coloneqq \varepsilon|_{C}$ be the counit of the central Hopf subalgebra $C$. Then $\mathfrak{m}_{\ovep}\coloneqq \text{Ker}\ovep\in \operatorname{maxSpec}C$ is the identity element of the algebraic group $\operatorname{maxSpec}C$. Moreover, the coproduct of $H$ turns the \textit{identity fiber algebra} $H/\mathfrak{m}_{\ovep}H$ into a finite-dimensional Hopf algebra, and endows each fiber algebra $H/\mathfrak{m}H$ with an $H/\mathfrak{m}_{\ovep}H$-$H/\mathfrak{m}_{\ovep}H$ bicomodule algebra structure (see \S \ref{subsec-fialgbiG}). In many examples of Cayley-Hamilton Hopf algebras arising from big quantum groups at roots of unity, the identity fiber algebras are isomorphic to the corresponding small quantum groups at roots of unity. 


One perspective to study the irreducible representations of $H$ is \textit{tensor-categorical}. Since $H$ is a noetherian affine PI $\mathbbm{k}$-algebra, the antipode of $H$ is bijective \cite[Corollary 2]{MR2271358} and all irreducible representations of $H$ are finite-dimensional over $\mathbbm{k}$ \cite[Theorem 13.10.3 (i)]{MR1811901}. Consequently, the irreducible representations of $H$ correspond exactly to the irreducible/simple objects in the tensor category $H\text{-mod}$ of finite-dimensional $H$-modules. As mentioned, each fiber algebra $H/\mathfrak{m}H$ is an $H/\mathfrak{m}_{\ovep}H$-comodule algebra. It follows that $H/\mathfrak{m}H\text{-mod}$ is a module category over the finite tensor category $H/\mathfrak{m}_{\ovep}H\text{-mod}$. In fact, by \cite[Proposition 2.8]{HMQWChev2025}, $H/\mathfrak{m}H\text{-mod}$ is indecomposable exact as a module category over the finite tensor category $H/\mathfrak{m}_{\ovep}H\text{-mod}$. Hence the Grothendieck group $\text{Gr}(H/\mathfrak{m}H)$ of $H/\mathfrak{m}H\text{-mod}$ is an irreducible $\mathbb{Z}_{+}$-module over the Grothendieck ring $\text{Gr}(H/\mathfrak{m}_{\ovep}H)$ by \cite[Proposition 7.7.2]{MR3242743}. In \cite{HMQWChev2025}, the authors build upon these observations to investigate the tensor action of irreducible representations of $H/\mathfrak{m}_{\ovep}H$ on other irreducible representations of $H$ using the Frobenius-Perron theory, aiming to understand how the structure of the identity fiber algebra $H/\mathfrak{m}_{\ovep}H$ affects the representation theory of $H$. Inspired by this insight, we consider irreducible representations whose tensor products with all irreducible representations satisfy certain nondegeneracy conditions. An $H$-module $V$ is called \textit{left tensor-reducible} (resp. \textit{right tensor-reducible}), if $V\otimes W$ (resp. $W\otimes V$) is completely reducible for any irreducible $H$-module $W$. An $H$-module is said to be \textit{tensor-reducible}, if it is both left tensor-reducible and right tensor-reducible. Clearly, all $1$-dimensional $H$-modules are tensor-reducible, and any left or right tensor-reducible module is itself completely reducible (see \S \ref{sub-tensred}). We regard the tensor-reducible representations of $H$ as the ``most nondegenerate'' representations of $H$ with respect to the tensor product. Moreover, it is clear that $H$ has the Chevalley property if and only if all irreducible $H$-modules are tensor-reducible. Thus, the study of tensor-reducible irreducible $H$-modules is essential to understanding the tensor-nondegenerate behaviour of irreducible representations of $H$, and the tensor-reducible property is closely tied to whether $H$ has the Chevalley property. 


Another perspective to study the irreducible representations of $H$ is to investigate the \textit{discriminant ideals} \cite{MR1972204,MR3415697} of $(H,C,\text{tr})$, as exploited in \cite{MR3886192,mi2025lowest,HMQWChev2025} recently. 
Discriminants and discriminant ideals are powerful tools in both commutative algebra and noncommutative algebra. Over the past decade, discriminants of noncommutative algebra and their applications have been extensively studied, see for example \cite{MR3281142,MR3415697,MR3720799,MR3663593,MR3886192,MR4843586}. However, very little work has been done on discriminant ideals. For any positive integer $k$, the \textit{$k$-discriminant ideal} $D_{k}(H/C;\text{tr})$ is the ideal of $C$ generated by
$$\{\text{det}(\text{tr}(h_{i}h_{j}))_{i,j=1}^{k}\mid (h_{1},...,h_{k})\in H^{k}\},$$
and the \textit{$k$-modified discriminant ideal} $MD_{k}(H/C;\text{tr})$ is the ideal of $C$ generated by
$$\{\text{det}(\text{tr}(h_{i}g_{j}))_{i,j=1}^{k}\mid (h_{1},...,h_{k}),(g_{1},...,g_{k})\in H^{k}\}.$$
Clearly, $MD_{k}(H/C;\text{tr})\supseteq D_{k}(H/C;\text{tr})$ and $MD_{k}(H/C;\text{tr})\supseteq MD_{k+1}(H/C;\text{tr})$ for all positive integers $k$. A discriminant ideal or modified discriminant ideal of $(H,C,\text{tr})$ is said to be \textit{trivial} if it is $0$ or $C$. Since $H$ is a finitely generated $C$-module, there is a positive integer $m$ 
 such that $MD_{m}(H/C;\text{tr})=MD_{m+1}(H/C;\text{tr})=\cdots=0$, which implies that there are only finitely many non-trivial (modified) discriminant ideals.

 By the Brown-Yakimov Theorem \cite[Theorem 4.1(b)]{MR3886192}, for any positive integer $k$,
\begin{equation}\label{eq-intro-discid}
\begin{aligned}
\mathcal{V}_{k}\coloneqq \mathcal{V}(D_{k}(H/C;\text{tr}))=&\mathcal{V}(MD_{k}(H/C;\text{tr}))\\
=& \{\mathfrak{m}\in \operatorname{maxSpec}C\mid \sum_{[V]\in \text{Irr}(H/\mathfrak{m}H)}(\dim_{\mathbbm{k}}V)^{2}<k \}.
\end{aligned}
\end{equation}
Following \cite{MR1103601}, the zero set $\mathcal{V}_{k}$ is called the \textit{$k$-discriminant subvariety} of $\operatorname{maxSpec}C$.
Brown-Yakimov's equality \eqref{eq-intro-discid} gives a deep relation between the dimensions of the irreducible representations of $H$ and the (modified) discriminant ideals of $(H,C,\text{tr})$. Moreover, there is a unique positive integer $\ell\leq m$ such that
$$\varnothing=\mathcal{V}_{1}=\mathcal{V}_{2}=\cdots=\mathcal{V}_{\ell-1}\subsetneq\mathcal{V}_{\ell}\subseteq \mathcal{V}_{\ell+1}\subseteq\cdots\subseteq \mathcal{V}_{m-1}\subseteq\mathcal{V}_{m}=\operatorname{maxSpec}C.$$
The corresponding discriminant ideal $D_{\ell}(H/C;\text{tr})$ (resp. the modified discriminant ideal $MD_{\ell}(H/C;\text{tr})$) is called the \textit{lowest discriminant ideal} (resp. the \textit{lowest modified discriminant ideal}) \cite{mi2025lowest}, and in this case the closed subvariety $\mathcal{V}_{\ell}$ is referred to as the \textit{lowest discriminant subvariety}. The fiber algebras associated with the maximal ideal in the lowest discriminant subvariety $\mathcal{V}_{\ell}$ may be regarded as the most degenerate among all fiber algebras and are also of particular interest for study. For instance, consider the quantized (positive) Borel subalgebra $\mathcal{U}_{\epsilon}^{\geq 0}(\mathfrak{g})$ of the De Concini-Kac quantized enveloping algebra $\mathcal{U}_{\epsilon}(\mathfrak{g})$ at a root of unity $\epsilon$ and the De Concini-Kac-Procesi central Hopf subalgebra $\mathcal{C}_{\epsilon}^{\geq 0}(\mathfrak{g})$ \cite{MR1351503}, see \S \ref{subsec-Uqb} for details. Then $\mathcal{U}_{\epsilon}^{\geq 0}(\mathfrak{g})$ is a finitely generated free $\mathcal{C}_{\epsilon}^{\geq 0}(\mathfrak{g})$-module, and $(\mathcal{U}_{\epsilon}^{\geq 0}(\mathfrak{g}),\mathcal{C}_{\epsilon}^{\geq 0}(\mathfrak{g}),\text{tr}_{\text{reg}})$ is an affine Cayley-Hamilton Hopf algebra \cite[Theorem 5.5]{mi2025lowest}, where $\text{tr}_{\text{reg}}:\mathcal{U}_{\epsilon}^{\geq 0}(\mathfrak{g})\to \mathcal{C}_{\epsilon}^{\geq 0}(\mathfrak{g})$ denotes the regular trace map (see Example \ref{eg-reutr1}). Moreover, all fiber algebras $\mathcal{U}_{\epsilon}^{\geq 0}(\mathfrak{g})/\mathfrak{m}\mathcal{U}_{\epsilon}^{\geq 0}(\mathfrak{g})$ over the lowest discriminant subvariety $\mathcal{V}(D_{\ell}(\mathcal{U}_{\epsilon}^{\geq 0}(\mathfrak{g})/\mathcal{C}_{\epsilon}^{\geq 0}(\mathfrak{g});\text{tr}_{\text{reg}}))$ are isomorphic to the small quantum Borel subgroup $ \mathfrak{u}_{\epsilon}^{\geq 0}(\mathfrak{g})$ as algebras. In \cite{mi2025lowest}, under the condition that the identity fiber algebra $H/\mathfrak{m}_{\ovep}H$ is basic, Mi-Wu-Yakimov carried out a detailed study of the lowest discriminant subvariety $V_{\ell}$, and provided a discription of $V_{\ell}$ in terms of maximally stable modules. Subsequently, for any affine Cayley-Hamilton Hopf algebra $(H,C,\text{tr})$ whose identity fiber has the Chevalley property, Huang-Mi-Qi-Wu \cite{HMQWChev2025} adopted a tensor-categorical approach and applied Frobenius-Perron theory successfully to provide several equivalent descriptions of the lowest discriminant subvariety, thereby extending several results in \cite{mi2025lowest}. Note that all the De Concini-Kac-Procesi big quantized Borel subalgebras at roots of unity \cite{MR1351503,MR1288995}, the De Concini-Lyubashenko quantized coordinate algebras at roots of unity \cite{MR1296515,MR1863398}, group algebras of central extensions of finite groups by finitely generated abelian groups, as well as many Artin-Schelter Gorenstein Hopf algebras of low GK dimension \cite{MR2661247,MR4281365,MR4042591}, fit into this framework \cite{mi2025lowest,HMQWChev2025}. 




This paper is devoted to understanding the interplay between the tensor nondegeneracy of irreducible representations (particularly the Chevalley property) and the discriminant ideals of Cayley-Hamilton Hopf algebras. 



The first result in this paper, Theorem \ref{mainthm-tensred4}, consists of three parts, proved in Theorems \ref{thm-tensred-insec2} and \ref{thm-whenalltriepare1d}. In the first part, under mild assumptions, it is proved that, for any irreducible representation of an affine Cayley-Hamilton Hopf algebra $(H,C,\text{tr})$, the conditions ``tensor-reducible'', ``left tensor-reducible'', and ``right tensor-reducible'' are equivalent, and that they characterize precisely the irreducible representations annihilated by the lowest discriminant ideal $D_{\ell}(H/C;\text{tr})$. This establishes a connection between the tensor-nondegenerate behaviour of $H$ and its lowest discriminant ideal. In the second part, the relationship between the tensor-reducibility of irreducible $H$-modules and the maximally stable property in the sense of Mi-Wu-Yakimov \cite{mi2025lowest} is fully elucidated. In the third part, we use the winding automorphisms of $H$ (see \S \ref{subsec-pfofthm1} for the definition) to characterize when the tensor-reducible irreducible $H$-modules are precisely the $1$-dimensional modules.

\begin{theorem}
[Theorems \ref{thm-tensred-insec2} and \ref{thm-whenalltriepare1d}]\label{mainthm-tensred4} Let $H$ be an affine Hopf algebra over an algebraically closed field $\mathbbm{k}$ of characteristic zero and let $C$ be a central Hopf subalgebra of $H$ such that $H$ is a finitely generated $C$-module and $H/\mathfrak{m}_{\overline{\varepsilon}}H$ has the Chevalley property. 
\begin{itemize}
\item[(a)] Suppose $\text{tr}: H \to C$ is a trace map such that $(H, C, \text{tr})$ is a Cayley-Hamilton Hopf algebra. Then for any irreducible $H$-module $V$, the following are equivalent.
\begin{itemize}
\item[(i)] $V$ is tensor-reducible;
\item[(ii)] $V$ is left tensor-reducible;
\item[(iii)] $V$ is  right tensor-reducible;
\item[(iv)] $V$ is annihilated by the lowest discriminant ideal $D_{\ell}(H/C;\text{tr})$;
\item[(v)] $V$ is annihilated by the lowest modified discriminant ideal $MD_{\ell}(H/C;\text{tr})$;
\item[(vi)] $V\otimes V^{*}$ is completely reducible.
\end{itemize}
\item[(b)]All maximally stable irreducible $H$-modules in the sense of \cite{mi2025lowest} are tensor-reducible. The converse holds provided that $H/\mathfrak{m}_{\ovep}H$ is basic.
\item[(c)]The following are equivalent.
\begin{itemize}
\item[(i)]All tensor-reducible irreducible $H$-modules are $1$-dimensional;
\item[(ii)]The identity fiber algebra $H/\mathfrak{m}_{\ovep}H$ is basic and all maximally stable irreducible $H$-modules are $1$-dimensional;
\item[(iii)]The identity fiber algebra $H/\mathfrak{m}_{\ovep}H$ is basic and 
$$\mathcal{V}_{\ell}=W_{l}(G(H^{\circ}))(\mathfrak{m}_{\ovep})=W_{r}(G(H^{\circ}))(\mathfrak{m}_{\ovep}),$$
where $W_{l}(G(H^{\circ}))$ and $W_{r}(G(H^{\circ}))$ are the left and right winding automorphism groups of $H$, respectively.
\end{itemize}
\end{itemize}
\end{theorem}

In \cite[Theorem 4.1]{HMQWChev2025}, the authors proved that the positive integer $\ell$ appearing in Theorem \ref{mainthm-tensred4} is exactly equal to $\text{FPdim}(\text{Gr}(H/\mathfrak{m}_{\ovep}H))+1$, where $\text{FPdim}(\text{Gr}(H/\mathfrak{m}_{\ovep}H))$ is the Frobenius-Perron dimension of the Grothendieck ring of $H/\mathfrak{m}_{\ovep}H\text{-mod}$. The implication (ii) $\Rightarrow$ (iii) in Theorem \ref{mainthm-tensred4} (c) was established in \cite[Theorem 4.3 (c)]{mi2025lowest}. Therefore, Theorem \ref{mainthm-tensred4} (c) may be regarded as providing a necessary and sufficient condition, in terms of the tensor-reducible irreducible representations of $H$, for the conditions of \cite[Theorem 4.3 (c)]{mi2025lowest} to hold.

Theorem \ref{mainthm-tensred4} shows that, for an affine Cayley-Hamilton Hopf algebra whose identity fiber algebra has the Chevalley property, determining its tensor-reducible irreducible representations reduces to determining the irreducible representations of the fiber algebras over the lowest discriminant subvarieties. Thanks to Mi-Wu-Yakimov's description \cite[Theorem 5.5(b)]{mi2025lowest} of the lowest discriminant varieties for the De Concini-Kac-Procesi big quantized Borel subalgebras $\mathcal{U}_{\epsilon}^{\geq 0}(\mathfrak{g})$ at roots of unity $\epsilon$, Theorem \ref{mainthm-tensred4} is used to deduce that the tensor-reducible irreducible representations of $\mathcal{U}_{\epsilon}^{\geq 0}(\mathfrak{g})$ are precisely the $1$-dimensional ones, see Proposition \ref{prop-quaborel-rofutenr}.


Moreover, Theorem \ref{mainthm-tensred4} implies that the tensor-nondegenerate behaviour of irreducible representations of $H$ is closely related to the lowest discriminant ideal of $(H,C,\text{tr})$. It is therefore natural to investigate the connection between the Chevalley property of $H$ and its discriminant ideals. In \cite{HMQWChev2025}, the authors showed that if $H$ has the Chevalley property, then all the discriminant ideals of $(H,C,\text{tr})$ are trivial, and posed the following question:
\begin{que}
Is there a necessary and sufficient condition for all affine Cayley-Hamilton
Hopf algebras to have the Chevalley property based on the properties of discriminant ideals?
\end{que}





The second result in this paper, Theorem \ref{intro-thm-iffchev}, gives an affirmative answer to the above question.
\begin{theorem}
[Theorem \ref{thm-Cheviff-ap} and Proposition \ref{prop-semiprimeChev}]\label{intro-thm-iffchev}Let $H$ be an affine Hopf algebra over an algebraically closed field $\mathbbm{k}$ of characteristic zero and let $C$ be a central Hopf subalgebra of $H$ such that $H$ is a finitely generated $C$-module.
\begin{itemize}
\item[(a)]For any trace map $\text{tr}:H\to C$ such that $(H,C,\text{tr})$ is a Cayley-Hamilton Hopf algebra, the following are equivalent:
\begin{itemize}
\item[(i)] $H$ has the Chevalley property;
\item[(ii)] The identity fiber algebra $H/\mathfrak{m}_{\ovep}H$ has the Chevalley property and the lowest discriminant ideal $D_{\ell}(H/C;\text{tr})=0$;
\item[(iii)] The identity fiber algebra $H/\mathfrak{m}_{\ovep}H$ has the Chevalley property and the lowest modified discriminant ideal $MD_{\ell}(H/C;\text{tr})=0$;
\item[(iv)]The identity fiber algebra $H/\mathfrak{m}_{\ovep}H$ has the Chevalley property and all the discriminant ideals of $(H,C,\text{tr})$ are trivial;
\item[(v)]The identity fiber algebra $H/\mathfrak{m}_{\ovep}H$ has the Chevalley property and all the modified discriminant ideals of $(H,C,\text{tr})$ are trivial.
\end{itemize}
\item[(b)] The following are equivalent for the pair $(H,C)$:
\begin{itemize}
\item[(i)] $H$ has the Chevalley property and finite global dimension;
\item[(ii)] $H$ is semiprime and has the Chevalley property;
\item[(iii)] The identity fiber algebra $H/\mathfrak{m}_{\ovep}H$ is semisimple.
\end{itemize}
\end{itemize}
\end{theorem}

By \cite[Theorem 2.9 (8)]{MR4201485}, for any affine Hopf algebra $H$ admitting a large central Hopf subalgebra $C$, if the identity fiber algebra $H/\mathfrak{m}_{\ovep}H$ is semisimple, then $H$ is regular. The converse is not true in general. For example, the quantized Borel subalgebra $\mathcal{U}_{\epsilon}^{\geq 0}(\mathfrak{g})$ is skew Calabi-Yau, while its identity fiber with respect to the De Concini-Kac-Procesi central Hopf subalgebra $\mathcal{C}_{\epsilon}^{\geq 0}(\mathfrak{g})$ is not semisimple. Hence the implication (i) $\Rightarrow$ (iii) in in Theorem \ref{intro-thm-iffchev} (b) essentially requires that $H$ has the Chevalley property.





 
The third result in this paper, Theorem \ref{thm-intrsubgrplv}, addresses the \textit{rigidity} of the lowest discriminant subvarieties of Cayley-Hamilton Hopf algebras. More precisely, we prove that, under the hypotheses of Theorem \ref{mainthm-tensred4}, the lowest discriminant subvariety $\mathcal{V}_{\ell}$ of $(H,C,\text{tr})$ is a subgroup of the algebraic group $\operatorname{maxSpec}C$. In particular, $\mathcal{V}_{\ell}$ is a smooth equidimensional subvariety of $\operatorname{maxSpec}C$. This feature is not shared by other discriminant subvarieties, see Example \ref{eg-OepSL2-notsubgrp}. Furthermore, we show that the lowest discriminant ideal of $(H,C,\text{tr})$ induces a (not necessary finite-dimensional) Hopf quotient of $H$ satisfying the Chevalley property. Thus, Theorem \ref{thm-intrsubgrplv}, on one hand, reflects the rigid nature of $\mathcal{V}_{\ell}$, and in some examples of low GK dimension, allows the closed subvariety $\mathcal{V}_{\ell}$ to be determined efficiently; on the other hand, it provides a method for using the lowest discriminant ideals to construct a large family of Hopf algebras, often infinite-dimensional, with the Chevalley property.
In particular, our results produce (non-finite) tensor categories with the Chevalley property from some big quantum groups at roots of unity, see Example \ref{eg-chevtensorcat-fromquanBorel}.



\begin{theorem}
[Theorem \ref{thm-lowestsetsubgrp} and Corollary \ref{cor-affCHquodisHq-cheviffid}]\label{thm-intrsubgrplv}Let $(H,C,\text{tr})$ be an affine Cayley-Hamilton Hopf algebra satisfying the conditions in Theorem \ref{mainthm-tensred4}.  Then the following hold.
\begin{itemize}
\item[(a)] The lowest discriminant subvariety $\mathcal{V}_{\ell}$ of $(H,C,\text{tr})$ is a closed subgroup of the algebraic group $\operatorname{maxSpec}C$. In particular, $\mathcal{V}_{\ell}$ is both smooth and equidimensional.
\item[(b)] Let $\mathfrak{d}_{\ell}= \sqrt{D_{\ell}(H/C;\text{tr})}$ be the radical of the lowest discriminant ideal $ D_{\ell}(H/C;\text{tr})$ in $C$ and let $\text{tr}_{\mathfrak{d}_{\ell}}:H/\mathfrak{d}_{\ell}H\to C/\mathfrak{d}_{\ell}$ be the induced trace map by $\text{tr}$.
\begin{itemize}
\item[(i)] The algebra with trace $(H/\mathfrak{d}_{\ell}H, C/\mathfrak{d}_{\ell}, \text{tr}_{\mathfrak{d}_{\ell}})$ is an affine Cayley-Hamilton Hopf algebra with the Chevalley property.
\item[(ii)] For any Hopf ideal $\mathfrak{a}$ of $C$, the Hopf quotient $H/\mathfrak{a}H$ has the Chevalley property if and only if $\mathfrak{a}\supseteq \mathfrak{d}_{\ell}$.
\end{itemize}
\end{itemize}
\end{theorem}

 If, in addition, $H$ is prime and of GK dimension one, then either $H$ is commutative or the lowest discriminant subvariety of $(H,C,\text{tr})$ is a finite cyclic subgroup of the algebraic group $\operatorname{maxSpec}C$, see Corollary \ref{cor-GK1prim-low}. Moreover,  every finite cyclic group arises as the lowest discriminant subvariety of some affine prime Cayley-Hamilton Hopf algebra of GK dimension one whose identity fiber algebra has the Chevalley property, see Example \ref{eg-geLiualg}.


The paper is organized as follows. In \S \ref{section 1}, we briefly present some preliminaries concerning module-finite algebras, Cayley-Hamilton algebras, and Hopf-Galois extensions. In \S \ref{sec-2}, we first use the fact that all fiber algebras of an affine Cayley-Hamilton Hopf algebra are bi-Galois objects over the identity fiber algebra to lay the groundwork for the proof of Theorem \ref{mainthm-tensred4}; after establishing some basic properties of tensor-reducible representations, we then provide the full proof of Theorem \ref{mainthm-tensred4} and derive Proposition \ref{prop-quaborel-rofutenr}.  In \S \ref{sec-3}, we prove Theorems \ref{intro-thm-iffchev} and \ref{thm-intrsubgrplv} and illustrate them with several examples.

\section{Preliminaries}\label{section 1}


Throughout the paper, $\mathbbm{k}$ is a fixed base field and all algebras are over $\mathbbm{k}$. An unadorned $\otimes$ signifies $\otimes_{\mathbbm{k}}$. All modules over algebras considered here are left modules. For an algebra $A$, the category of all finite-dimensional $A$-modules is denoted by $A\text{-mod}$. The set of isomorphism classes of irreducible $A$-modules is denoted by $\text{Irr}(A)$ and the maximal spectrum of $A$ is denoted by $\operatorname{maxSpec}A$. For any $A$-module $W$ of finite length and any irreducible $A$-module $V$, $[W:V]$ is the multiplicity of $V$ in a composition series of $W$. 

We refer to the reader \cite{MR1243637} for some basic materials about Hopf algebras. The coproduct, counit, and antipode of a Hopf algebra are denoted by $\Delta,\varepsilon$ and $S$, respectively. For any Hopf algebra $H$ and $h\in H$, we use Sweedler's notation $\Delta(h)=\sum_{(h)} h_{(1)}\otimes h_{(2)} \in H \otimes H$. We denote the finite dual Hopf algebra of $H$ by $H^{\circ}$. The set of group-like elements of $H$ is denoted by $G(H)$. It is clear that $G(H^{\circ})$ coincides with the set of all characters of $H$.

In this section, we provide some preliminaries on module-finite algebras, Cayley-Hamilton algebras, and Hopf-Galois extensions. For more detailed background, we refer the reader to our primary references \cite{MR1898492,MR1811901,MR883869,MR1288995,MR2178656,MR1408508,MR2075600}.

\subsection{Module-finite algebras}\label{sec-modufinalg}
Let $A$ be an algebra and let $C$ be a central subalgebra of $A$. We say that $A$ is a \textit{module-finite $C$-algebra} \cite[p.xv]{MR2080008}, if $A$ is a finitely generated $C$-module. In this subsection, we assume that \textit{$A$ is an affine module-finite $C$-algebra} over an algebraically closed field $\mathbbm{k}$. By the Artin-Tate Lemma, $C$ is an affine $\mathbbm{k}$-algebra. Since $A$ is an affine PI algebra over $\mathbbm{k}$, all irreducible $A$-module are finite-dimensional over $\mathbbm{k}$ \cite[Theorem 13.10.3 (i)]{MR1811901}. For any $\mathfrak{m}\in \operatorname{maxSpec}C$, $C/\mathfrak{m}\cong \mathbbm{k}$ by Hilbert's Nullstellensatz. Hence $A/\mathfrak{m}A$ is finite-dimensional over $\mathbbm{k}$. The family of finite-dimensional algebras $\{A/\mathfrak{m}A\mid \mathfrak{m}\in \operatorname{maxSpec}C \}$ \textit{contains all information about the irreducible representations of $A$.} Indeed, by Kaplansky's Theorem, the map $\text{Irr}(A)\to\operatorname{maxSpec}A, [V]\mapsto \text{Ann}_{A}V$ is a well-defined bijection. Moreover, the map $\pi:\operatorname{maxSpec}A\to \operatorname{maxSpec}C, Q\mapsto Q\cap C$ is surjective, by \cite[Theorem 13.8.14]{MR1811901}. Hence, for any irreducible $A$-module $V$, $V$ is an irreducible module over the finite dimensional algebra $A/\mathfrak{m}A$, where $\mathfrak{m}=(\text{Ann}_{A}V)\cap C \in \operatorname{maxSpec}C$. Note that for any $\mathfrak{m}\in \operatorname{maxSpec}C$, the set of isomorphism classes of irreducible $A$-modules corresponding to the fiber $\pi^{-1}(\mathfrak{m})$ of $\mathfrak{m}$ under $\pi$ is precisely $\text{Irr}(A/\mathfrak{m}A)$. The finite-dimensional algebra $A/\mathfrak{m}A$ is called the \textit{fiber algebra} of $A$ at $\mathfrak{m}$.

\subsection{Cayley-Hamilton algebras}
Throughout this subsection, we assume that the base field $\mathbbm{k}$ is of characteristic zero. Let $A$ be a $\mathbbm{k}$-algebra and let $C$ be a central subalgebra of $A$. A $C$-linear map $\text{tr}:A\to C$ is called a \textit{trace map} \cite{MR883869,MR3886192} if it  satisfies $\text{tr}(ab)=\text{tr}(ba)$ for all $a,b\in A$. The triple $(A,C,\text{tr})$ is called an \textit{algebra with trace} \cite{MR883869}.

Let $(A,C,\text{tr})$ be a $\mathbbm{k}$-algebra with trace, and let $n$ be a positive integer. For any positive integer $1\leq k\leq n$ and $a\in A$, let 
$$c_{k}(a)=\frac{1}{k!}\left|\begin{array}{ccccc}\text{tr}(a) & 1 & 0 & \cdots & 0 \\ \text{tr}(a^{2}) & \text{tr}(a) & 2 & \cdots & 0 \\ \vdots & \vdots & \vdots & & \vdots \\ \text{tr}(a^{k-1}) & \text{tr}(a^{k-2}) & \text{tr}(a^{k-3}) & \cdots & k-1 \\ \text{tr}(a^{k}) & \text{tr}(a^{k-1}) & \text{tr}(a^{k-2}) & \cdots & \text{tr}(a)\end{array}\right|\in C.$$
Then, for any $a\in A$, one defines the \textit{$n$-th characteristic polynomial} of $a$ as
\begin{equation}\label{eq-charpoly-CHalg}
p_{n,a}(t)=t^{n}-c_{1}(a)t^{n-1}+\cdots+(-1)^{n-1}c_{n-1}(a)t+(-1)^{n}c_{n}(a)\in C[t].
\end{equation}

If $\text{tr}:\text{M}_{n}(C)\to C$ is the usual trace map of the matrix algebra over a commutative algebra $C$, then \eqref{eq-charpoly-CHalg} coincides with the usual characteristic polynomial of a matrix.

\begin{definition}
[\cite{MR883869,MR1288995}]Let $(A,C,\text{tr})$ be a $\mathbbm{k}$-algebra with trace, $n$ be a positive integer. The algebra with trace $(A,C,\text{tr})$ is called a Cayley-Hamilton algebra of degree $n$, if $p_{n,a}(a)=0$ for all $a\in A$, and $\text{tr}(1)=n$.
\end{definition}
By Cayley-Hamilton theorem, $(\text{M}_{n}(C),C,\text{tr})$ is a Cayley-Hamilton algebra of degree $n$. 

\begin{example}
[regular trace]Suppose $A$ is a module-finite $C$-algebra and  $_{C}A$ is a free module of rank $n$.\label{eg-reutr1} Let $\{b_{1},...,b_{n}\}$ be a $C$-basis of $A$. Then for any element $a\in A$, there is a unique matrix $\iota(a)\in \text{M}_{n}(C)$ such that $(ab_{1},...,ab_{n})=(b_{1},...,b_{n})\iota(a)$. This gives an algebra embedding $\iota:A\to\text{M}_{n}(C)$. Then $\text{tr}_{\text{reg}}(a)=\text{tr} (\iota(a))$ gives a trace map $\text{tr}_{\text{reg}}: A\to C$. It is clear that the trace map $\text{tr}_{\text{reg}}:A\to C$ is independent of the choices of the $C$-basis of $A$, which is called the \textit{regular trace} of $A$. It is direct to verify that the algebra with trace $(A,C,\text{tr}_{\text{reg}})$ is a Cayley-Hamilton algebra of degree $n$.
\end{example}
\begin{example}
[Hattori-Stallings trace, \cite{MR175950,MR202807}]Let $A$ be a module-finite $C$-algebra\label{eg-HStrace} such that $_{C}A$ is projective. Let $\{x_{1},...,x_{m}\}\subseteq A$ and $\{x_{1}^{*},...,x_{m}^{*}\}\subseteq \text{Hom}_{C}(A,C)$ be a dual basis of the projective $C$-module $A$. Then one can directly verify that the map $\text{tr}_{HS}:A\to C$ defined by
$$\text{tr}_{HS}(a)=\sum_{k=1}^{m}x_{k}^{*}(ax_{k})$$ 
is a well-defined trace map, which is independent of the choice of dual basis of $_{C}A$. In \cite[Theorem 2.7]{HMQWChev2025}, it was proved that if $A$ is a projective $C$-module of rank $n$, then $(A,C,\text{tr}_{HS})$ is a Cayley-Hamilton algebra of degree $n$. It is worth noting that if $A$ is a free $C$-module, then $\text{tr}_{HS}$ coincides with the regular trace given in Example \ref{eg-reutr1}.
\end{example}
\begin{definition}
[\cite{MR2178656}] A Cayley-Hamilton Hopf algebra is a Cayley-Hamilton algebra $(H,C,\text{tr})$ such that  $H$ is a Hopf algebra and $C$ is a central Hopf subalgebra of $H$.\label{def-CHHopf}
\end{definition}
For all affine Hopf algebras $H$ that admit large central Hopf subalgebras $C$, it was proved in \cite[Theorem 2.5]{HMQWChev2025} that $_{C}H$ are projective modules of constant rank and the algebras with trace $(H,C,\text{tr}_{HS})$ are Cayley-Hamilton Hopf algebras. 

\subsection{Hopf-Galois extensions}
Let $H$ be a Hopf algebra over the field $\mathbbm{k}$. Recall that for a right $H$-comodule algebra $A$, the extension $A^{coH}\subseteq A$ is said to be \textit{right $H$-Galois}, if the canonical map $\beta_{r}:A\otimes_{A^{coH}}A\to A\otimes H,a\otimes b\mapsto \sum_{(b)}ab_{(0)}\otimes b_{(1)} $ is bijective. If, in addition, $A^{coH}=\mathbbm{k}$, then $A$ is called a \textit{right $H$-Galois object}. Similarly, for a left comodule algebra, one can define the notions of a \textit{left Galois extension} and a \textit{left Galois object}. An $H$-$H$ bicomodule algebra is called an \textit{$H$-$H$-bi-Galois object} \cite{MR1408508,MR2075600}, if it is both a left and a right $H$-Galois object. Later we will need the following fact.
\begin{lemma}
[\text{\cite[Corollary 3.1.4]{MR2075600}}]Let $A$ be an $H$-$H$-bi-Galois object, and let $f:A\to A$ be an algebra homomorphism.  Suppose $f$ is right $H$-comodule homomorphism, then there is some $\theta\in G(H^{\circ})$ such that $f(a)=\sum_{(a)}\theta(a_{(-1)})a_{(0)}$ for all $a\in A$.\label{lem-ricomoen-wingdf-bgo}
\end{lemma}

\section{Tensor-reducible representations of Hopf algebras}\label{sec-2}
In this section, we first establish the groundwork required for the proof of Theorem \ref{mainthm-tensred4}. These results rely on the fact that, for any affine Hopf algebra admitting a large central Hopf subalgebra, every fiber algebra is a bi-Galois object over the identity fiber algebra. We then present several basic properties of tensor-reducible representations of such Hopf algebras and conclude with the proof of Theorem \ref{mainthm-tensred4}. In the remainder of this paper, the base field $\mathbbm{k}$ is assumed to be an algebraically closed field of characteristic zero, and we work under the following hypothesis:
\begin{align*}
\text{(FinHopf) }& H \textit{ is an affine Hopf algebra over  }\mathbbm{k}\textit{ with a } \textit{central Hopf subalgebra } C \\
& \textit{such that } H\textit{ is a finitely generated }C\textit{-module}.
\end{align*}

\subsection{Fiber algebras and bi-Galois objects}\label{subsec-fialgbiG}

Let $(H,C)$ be a pair of Hopf algebras satisfying the hypothesis (FinHopf). Then $H$ is a noetherian affine PI Hopf algebra. By \cite[Corollary 2]{MR2271358}, the antipode $S$ of $H$ is bijective. Hence $H\text{-mod}$ is a tensor category with respect to the tensor product.  For any finite-dimensional $H$-module $V$, the left dual $V^{\ast}$ and right dual $^{\ast}V$ of $V$ are the usual dual space of $V$, with actions of $H$ given respectively by
$$(h\xi)(v)=\xi(S(h)v)\text{ and\  } (h\xi)(v)=\xi(S^{-1}(h)v)$$
for all $h\in H,v\in V$ and $\xi\in \text{Hom}_{\mathbbm{k}}(V,\mathbbm{k})$. Clearly, if $V$ is an irreducible $H$-module, then both $V^{\ast}$ and $^{\ast}V$ are irreducible $H$-modules.

 In this setting, one also has an algebraic group isomorphism $G(C^{\circ})\cong \operatorname{maxSpec}C$ given by $\chi\mapsto \mathfrak{m}_{\chi}\coloneqq \text{Ker}\chi$. For $\chi,\psi\in G(C^{\circ})$, the convolution product of $\chi$ amd $\psi$ is denoted by $\chi\ast \psi$. Sometimes it will be convenient to write $\mathfrak{m}_{\chi\ast \psi}$ as $\mathfrak{m}_{\chi}\ast\mathfrak{m}_{\psi}$.

 The identity element of $G(C^{\circ})$ with respect to the convolution product is given by $\ovep=\varepsilon|_{C}:C\to\mathbbm{k}$. For any $\chi\in G(C^{\circ})$, the inverse of $\chi$ with respect to the convolution product is given by $\chi S$. Thus, the inverse of $\mathfrak{m}\in \operatorname{maxSpec}C$ is $\mathfrak{m}_{\chi S}$, denoted by $\mathfrak{m}^{-1}$. 
 
 For any two characters $\chi,\psi\in G(C^{\circ})$, it is clear that the coproduct $\Delta:H\to H\otimes H$ induces an algebra homomorphism
\begin{equation}\label{eq-Del-fibers}
\Delta_{\chi,\psi}:H/\mathfrak{m}_{\chi\ast \psi}H\to H/\mathfrak{m}_{\chi}H\otimes H/\mathfrak{m}_{\psi}H,\overline{h}\mapsto \sum_{(h)}\overline{h_{(1)}}\otimes \overline{h_{(2)}}.
\end{equation}
Thus, the identity fiber algebra $ H/\mathfrak{m}_{\ovep}H$ is a finite-dimensional Hopf algebra with coproduct $\Delta_{\ovep,\ovep}$. It follows from \eqref{eq-Del-fibers} that each fiber algebra $H/\mathfrak{m}H$ is a left and right comodule algebra over $ H/\mathfrak{m}_{\ovep}H$. Moreover, \eqref{eq-Del-fibers} defines a bifunctor
\begin{equation}\label{eq-bifunct-conv-fiberalgsten}
\otimes : H/\mathfrak{m}_{\chi}H\text{-mod} \times H/\mathfrak{m}_{\psi}H\text{-mod}\to H/\mathfrak{m}_{\chi\ast\psi}H\text{-mod}.
\end{equation}
In particular, for any $\mathfrak{m}\in \operatorname{maxSpec}C$, \eqref{eq-bifunct-conv-fiberalgsten} turns $H/\mathfrak{m}H\text{-mod}$ into both a left and a right module category over the finite tensor category $H/\mathfrak{m}_{\ovep}H\text{-mod}$. Moreover, all fiber algebras $H/\mathfrak{m}H$ are bi-Galois objects over the identity fiber algebra $H/\mathfrak{m}_{\ovep}H$ as proved in the following lemma, which is a special case of \cite[Lemma 2.4]{MR3530496}. 


\begin{lemma} \label{lem-biGalo-allfibalg1}
Let $(H,C)$ be a pair of Hopf algebras satisfying the hypothesis (FinHopf). Then $H/\mathfrak{m}H$ is an $H/\mathfrak{m}_{\ovep}H$-$H/\mathfrak{m}_{\ovep}H$-bi-Galois object  for any $\mathfrak{m}\in \operatorname{maxSpec}C$.
\end{lemma}
\begin{proof}
Clearly, $H/\mathfrak{m}H$ is an $H/\mathfrak{m}_{\ovep}H$-$H/\mathfrak{m}_{\ovep}H$ bicomodule algebra. We show that $H/\mathfrak{m}H$ is a right $H/\mathfrak{m}_{\ovep}H$-Galois object. The proof for the left case is symmetric.
It follows from the proof of \cite[Proposition 2.1 (1)]{HMQWChev2025} that $$\beta: H/\mathfrak{m}H\otimes H/\mathfrak{m}H\to  H/\mathfrak{m}H\otimes  H/\mathfrak{m}_{\ovep}H, \overline{h}\otimes \overline{g}\mapsto \sum_{(g)}\overline{h}\overline{g_{(1)}}\otimes \overline{g_{(2)}}$$ is a $\mathbbm{k}$-linear isomorphism. So it suffices to prove 
$(H/\mathfrak{m}H)^{co (H/\mathfrak{m}_{\ovep}H)}=\mathbbm{k}.$
Consider the maps
$$\xi:H/\mathfrak{m}H\to H/\mathfrak{m}H\otimes H/\mathfrak{m}H,\overline{h}\mapsto \overline{h}\otimes 1-1\otimes \overline{h}, \text{ and }$$
$$\zeta:H/\mathfrak{m}H\to H/\mathfrak{m}H\otimes H/\mathfrak{m}_{\ovep}H, \overline{h}\mapsto \overline{h}\otimes 1-\sum_{(h)}\overline{h_{(1)}}\otimes\overline{h_{(2)}}.$$
Obviously $\text{Ker}\,\xi=\mathbbm{k}$. It follows from the following commutative diagram
$$\begin{tikzcd}
H/\mathfrak{m}H \arrow[rr, "\xi"] \arrow[rd, "\zeta"'] & & H/\mathfrak{m}H\otimes H/\mathfrak{m}H \arrow[ld, "\beta"] \\
&H/\mathfrak{m}H\otimes H/\mathfrak{m}_{\ovep}H &  
\end{tikzcd}$$
and $\beta$ is an isomorphism that $\mathbbm{k}=\text{Ker}\,\xi=\text{Ker}\,\zeta=(H/\mathfrak{m}H)^{co (H/\mathfrak{m}_{\ovep}H)}$. 
\end{proof}
\begin{remark}
Let $\text{BiGal}(H/\mathfrak{m}_{\ovep}H)$ be the set of isomorphism classes of $H/\mathfrak{m}_{\ovep}H$-$H/\mathfrak{m}_{\ovep}H$-bi-Galois objects. By \cite[Theorem 4.3]{MR1408508}, $\text{BiGal}(H/\mathfrak{m}_{\ovep}H)$ is a group under the cotensor product $\square_{H/\mathfrak{m}_{\ovep}H}$. We remark that the natural map $\operatorname{maxSpec}C\to \text{BiGal}(H/\mathfrak{m}_{\ovep}H)$ given by $\mathfrak{m}\mapsto [H/\mathfrak{m}H]$ is a group homomorphism, see \cite[Corollary 2.6]{MR3530496}. 
\end{remark}
For convenience, in the remainder of this subsection we adopt the following notation. Let $\chi\in G(C^{\circ})$. The fiber algebra $H/\mathfrak{m}_{\chi}H$ will be abbreviated as $H_{\chi}$, its Jacobson radical will be denoted by $J_{\chi}$, and the quotient $H_{\chi}/J_{\chi}$ will be written as $\overline{H_{\chi}}$. The canonical projection from $H$ to $H_{\chi}$ will be denoted by $\pi_{\chi}:H\to H_{\chi}$.  Since $S(\mathfrak{m}_{\chi})=\mathfrak{m}_{\chi^{-1}}$, the antipode $S$ of $H$ induces an algebra anti-isomorphism $S_{\chi}:H_{\chi}\to H_{\chi^{-1}},\overline{h}\mapsto \overline{S(h)}$, and we denote by $\overline{S_{\chi}}$ the induced map from $\overline{H_{\chi}}$ to $\overline{H_{\chi^{-1}}}$.

Suppose the identity fiber algebra $H_{\ovep}$ has the Chevalley property, then by \cite[Theorem 2.1]{MR1995055} (see also the proof of \cite[Proposition 3.9]{HMQWChev2025}), for all $\chi\in G(C^{\circ})$ one has
\begin{equation}\label{eq-desbyjac-fibalgs1}
\Delta_{\chi,\ovep}(J_{\chi})\subseteq H_{\chi}\otimes J_{\ovep}+J_{\chi}\otimes H_{\ovep}\text{ and\ }\Delta_{\ovep,\chi}(J_{\chi})\subseteq H_{\ovep}\otimes J_{\chi}+J_{\ovep}\otimes H_{\chi}.
\end{equation}
Thus, under the assumption that $H_{\ovep}$ has the Chevalley property, it follows from \eqref{eq-desbyjac-fibalgs1} that for each $\chi\in G(C^{\circ})$, $\overline{H_{\chi}}$ is an $\overline{H_{\ovep}}$-$\overline{H_{\ovep}}$ bicomodule algebra. The corresponding comodule structure maps are denoted by $\overline{\Delta_{\chi,\ovep}}:\overline{H_{\chi}}\to \overline{H_{\chi}}\otimes \overline{H_{\ovep}}$ and $\overline{\Delta_{\ovep,\chi}}:\overline{H_{\chi}}\to \overline{H_{\ovep}}\otimes \overline{H_{\chi}}$. 

The main goal of this subsection is to show that \eqref{eq-desbyjac-fibalgs1} remains valid when $\ovep$ replaced by any character $\varphi\in G(C^{\circ})$ satisfying that $\dim_{\mathbbm{k}}\overline{H_{\varphi}}=\dim_{\mathbbm{k}}\overline{H_{\ovep}}$ (see Proposition \ref{thm-DelJinone-bardimeq1}). This will be a key ingredient in the proof of Theorem \ref{mainthm-tensred4}. We first need the following.
   
\begin{proposition}
Let $(H,C)$ be a pair of Hopf algebras satisfying the hypothesis (FinHopf) such that the identity fiber algebra $H_{\ovep}$ has the Chevalley property. Then, for any $\chi\in G(C^{\circ})$ satisfying 
 $\dim_{\mathbbm{k}}\overline{H_{\chi}}=\dim_{\mathbbm{k}}\overline{H_{\ovep}}$, the following hold.\label{prop-quobyJ-bgaloiobj-Chevs}
\begin{itemize}
\item[(a)] $\overline{H_{\chi}}$ is an $\overline{H_{\ovep}}$-$\overline{H_{\ovep}}$-bi-Galois object.
\item[(b)] There exists some $\theta\in G(H^{\circ})$ such that $\text{Ker } \theta \supseteq \m_{\ovep}$ and
$$\pi_{\chi}(S^{2}(h))=\overline{S^{2}(h)}=\sum_{(h)}\overline{\theta(h_{(1)})h_{(2)}}=\sum_{(h)}\theta(h_{(1)})\pi_{\chi}(h_{(2)})\in H_{\chi}$$
for all $h\in H$.
\end{itemize}
\end{proposition}
\begin{proof}
(a) By Lemma \ref{lem-biGalo-allfibalg1}, the fiber algebra $H_{\chi}$ is an $H_{\ovep}$-$H_{\ovep}$-bi-Galois object. It follows that the left Galois map $\beta_{l}:H_{\chi}\otimes H_{\chi}\to   H_{\ovep}\otimes H_{\chi}$ and the right Galois map $\beta_{r}:H_{\chi}\otimes H_{\chi}\to H_{\chi}\otimes H_{\ovep}$ induce respectively surjections
\begin{equation}\label{eq-galoimap-barn7}
\overline{\beta}_{l}:\overline{H_{\chi}}\otimes \overline{H_{\chi}}\to  \overline{H_{\ovep}}\otimes \overline{H_{\chi}}\text{ and}\    \overline{\beta}_{r}:\overline{H_{\chi}}\otimes \overline{H_{\chi}}\to \overline{H_{\chi}}\otimes \overline{H_{\ovep}}.
\end{equation}
Since $\dim_{\mathbbm{k}}\overline{H_{\chi}}=\dim_{\mathbbm{k}}\overline{H_{\ovep}}$, it is clear that both maps in \eqref{eq-galoimap-barn7} are isomorphisms. This implies that the bicomodule algebra $\overline{H_{\chi}}$ is an $\overline{H_{\ovep}}$-$\overline{H_{\ovep}}$-bi-Galois object.
\par(b) Since $\overline{H_{\ovep}}$ is a semisimple Hopf algebra over $\mathbbm{k}$, it follows from a classical result of Larson and Radford \cite{MR957441,MR926744} that $\overline{S_{\ovep}}^{2} = \text{id}$. Note that the following diagram commutes:
$$
\begin{tikzcd}
& \overline{H_{\chi}} \arrow[rr,"{\overline{\Delta}_{\chi,\ovep}}"] \arrow[ld,"\overline{S_{\chi}}"'] &  & \overline{H_{\chi}}
\otimes \overline{H_{\ovep}} \arrow[rd, "\overline{S_{\chi}} \otimes \overline{S_{\ovep}}"] &                             \\                                                           
\overline{H_{\chi^{-1}}} \arrow[rd,"\overline{S_{\chi^{-1}}}"'] &                               &  &
& \overline{H_{\chi^{-1}}} \otimes \overline{H_{\ovep}} \arrow[ld, "\overline{S_{\chi^{-1}}} \otimes \overline{S_{\ovep}}"] \\
   & \overline{H_{\chi}} \arrow[rr,"{\overline{\Delta_{\chi,\ovep}}}"]   &   &
\overline{H_{\chi}}\otimes \overline{H_{\ovep}}  &
\end{tikzcd}
$$
It follows that $\overline{S_{\chi^{-1}}} \overline{S_{\chi}}:\overline{H_{\chi}}\to \overline{H_{\chi}} $ is a right $\overline{H_{\ovep}}$-comodule algebra endomorphism of $\overline{H_{\chi}}$. Hence,  $\overline{H_{\chi}}$ is an $\overline{H_{\ovep}}$-$\overline{H_{\ovep}}$-bi-Galois object by (a). By Lemma \ref{lem-ricomoen-wingdf-bgo}, there is a character $\theta_{0}:\overline{H_{\ovep}}\to\mathbbm{k}$  such that 
$$\overline{S^{2}(h)} = \overline{S_{\chi^{-1}}} \overline{S_{\chi}}(\overline{h}) = \sum_{(h)} \theta_{0}(\overline{h_{(1)}}) \overline{h_{(2)}}$$
for all $h \in H.$ Then $\theta:H\to\mathbbm{k}, h\mapsto \theta_{0}(\overline{h})$ is the desired character.
\end{proof}

The following isomorphism will be used in the proof of Proposition \ref{thm-DelJinone-bardimeq1}. For any two finite-dimensional $\mathbbm{k}$-algebras $A,B$, one has a canonical isomorphism
\begin{equation} \label{a-cano-alg-iso}
    \Phi_{A,B}: A\otimes B\to \text{Hom}_{\mathbbm{k}}(B^{*},A),a\otimes b\mapsto \big(g\mapsto g(b)a\big)
\end{equation}
which is natural in each variable. If $\text{Hom}_{\mathbbm{k}}(B^{*},A)$ is regarded as a $\mathbbm{k}$-algebra under the convolution product, then it is easy to verify that $\Phi_{A,B}$ is an algebra isomorphism. 

Now, for any $\varphi,\psi\in G(C^{\circ})$, let $\Gamma_{\varphi\ast \psi}$ be the composition of the following maps:
$$\begin{tikzcd}
H_{\varphi\ast\psi} \arrow[r, "{\Delta_{\varphi,\psi}}"] & H_{\varphi}\otimes H_{\psi} \arrow[r, "\text{id}\otimes \pi_{\psi}"] & H_{\varphi}\otimes \overline{H_{\psi}} \arrow[r, "{\Phi_{H_{\varphi},\overline{H_{\psi}}}}"] & {\text{Hom}_{\mathbbm{k}}((\overline{H_{\psi}})^{*},H_{\varphi})}.
\end{tikzcd}$$
For any $h\in H$ and $f\in (\overline{H_{\psi}})^{*}$, then, by \eqref{a-cano-alg-iso}
\begin{equation}\label{eq-Gamactf1}
\Gamma_{\varphi\ast\psi}(\pi_{\varphi\ast\psi}(h))(f)=\sum_{(h)}\pi_{\varphi}(h_{(1)})f(\overline{\pi_{\psi}(h_{(2)})})\in H_{\varphi}.
\end{equation}
Since $(\overline{H_{\psi}})^{*}$ can be viewed naturally as a subcoalgebra of $(H_{\psi})^{*}$, sometimes we identify $(\overline{H_{\psi}})^{*}$ with its image in $(H_{\psi})^{*}$. Then,  each $f\in (\overline{H_{\psi}})^{*}$ induces a map via \eqref{eq-Gamactf1}:
$$(f\rightharpoonup -): H_{\varphi\ast\psi} \to H_{\varphi}, x \mapsto (f\rightharpoonup x),$$
where, for any $h\in H$,
 $$f\rightharpoonup \pi_{\varphi\ast \psi}(h) \coloneqq \sum_{(h)}\pi_{\varphi}(h_{(1)})f( \pi_{\psi}(h_{(2)})).$$ The following lemma will also be needed in the proof of Proposition \ref{thm-DelJinone-bardimeq1}.

\begin{lemma} \label{lem-Hpbarcoacthg2}
Assume that $(H,C)$ is a pair of Hopf algebras satisfying the hypothesis (FinHopf). Let $\varphi,\psi\in G(C^{\circ})$.
\begin{itemize}
\item[(a)] For any $h\in H$ and $f\in (\overline{H_{\psi}})^{*}$,
$$\sum_{(f)}f_{(1)}\rightharpoonup (f_{(2)}\overline{S_{\psi^{-1}}}\rightharpoonup  \pi_{\varphi}(h))=\pi_{\varphi}(h)\varepsilon_{(\overline{H_{\psi}})^{*}}(f).$$ 
\item[(b)] For any $h, k\in H$ and $f\in (\overline{H_{\psi}})^{*}$,
$$
\sum_{(f)}(f_{(1)}\rightharpoonup \pi_{\varphi\ast\psi}(h))(f_{(2)}\rightharpoonup \pi_{\varphi\ast\psi}(k))=f\rightharpoonup \pi_{\varphi\ast\psi}(hk).
$$
\end{itemize}
\end{lemma}
\begin{proof} (a) Since $\Delta(f)=\sum f_{(1)} \otimes f_{(2)} \in (H_{\psi})^{*} \otimes (H_{\psi})^{*}$ and $\pi_{\varphi}= \pi_{(\varphi\ast\psi)\ast\psi^{-1}}$, 
$$f_{(2)}\overline{S_{\psi^{-1}}}\in (\overline{H_{\psi^{-1}}})^{*}\text{ and\ }f_{(2)}\overline{S_{\psi^{-1}}}\rightharpoonup  \pi_{\varphi}(h)\in H_{\varphi\ast\psi}.$$
Then
\begin{align*}
&\sum_{(f)}f_{(1)}\rightharpoonup (f_{(2)}\overline{S_{\psi^{-1}}}\rightharpoonup  \pi_{\varphi}(h))\\
=& \sum_{(f),(h)}f_{(1)}\rightharpoonup  \pi_{\varphi\ast\psi}(h_{(1)})f_{(2)}(\pi_{\psi}(S(h_{(2)})))\\
=&\sum_{(f),(h)}\pi_{\varphi}(h_{(1)})f_{(1)}(\pi_{\psi}(h_{(2)}))f_{(2)}(\pi_{\psi}(S(h_{(3)})))\\
=&\sum_{(h)}\pi_{\varphi}(h_{(1)})f(\pi_{\psi}(h_{(2)})\pi_{\psi}(S(h_{(3)})))\\
=&\pi_{\varphi}(h)f(\pi_{\psi}(1))\\
=&\pi_{\varphi}(h)\varepsilon_{(\overline{H_{\psi}})^{*}}(f).
\end{align*}
So, (a) holds. 

(b) follows a direct computation with the same style as in (a).
\end{proof}

\begin{proposition}
Let $(H,C)$ be a pair of Hopf algebras satisfying the hypothesis (FinHopf) such that the identity fiber algebra $H_{\ovep}$ has the Chevalley property. 
Suppose $\varphi,\chi\in G(C^{\circ})$, and either $\dim_{\mathbbm{k}}\overline{H_{\varphi}}=\dim_{\mathbbm{k}}\overline{H_{\ovep}}$ or $\dim_{\mathbbm{k}}\overline{H_{\chi}}=\dim_{\mathbbm{k}}\overline{H_{\ovep}}$ holds. Then
\begin{itemize}
\item[(a)]   
$\Delta_{\varphi,\chi}(J_{\varphi\ast\chi})\subseteq J_{\varphi}\otimes H_{\chi}+H_{\varphi}\otimes J_{\chi}.$ \label{thm-DelJinone-bardimeq1} 
\item[(b)] 
for any irreducible $H_{\varphi}$-module $V$ and any irreducible $H_{\chi}$-module $W$, $V\otimes W$ is completely reducible as an $H$-module .
\end{itemize}
\end{proposition}
\begin{proof}

(a) We claim that, if $\dim_{\mathbbm{k}}\overline{H_{\chi}}=\dim_{\mathbbm{k}}\overline{H_{\ovep}}$, then for any $\varphi\in G(C^{\circ})$, one has
\begin{equation}\label{eq-clDechiu1}
(\overline{H}_{\chi})^{*}\rightharpoonup J_{\varphi\ast\chi}\subseteq J_{\varphi}.
\end{equation}
Indeed, once \eqref{eq-clDechiu1} is established, it follows from the following commutative diagram that $\Delta_{\varphi,\chi}(J_{\varphi,\chi})\subseteq \text{Ker}(\pi_{\varphi}\otimes \pi_{\chi})$. Consequently, we obtain that $\Delta_{\varphi,\chi}(J_{\varphi\ast\chi})\subseteq J_{\varphi}\otimes H_{\chi}+H_{\varphi}\otimes J_{\chi}.$
$$\begin{tikzcd}
H_{\varphi}\otimes \overline{H_{\chi}} \arrow[d, "\pi_{\varphi}\otimes \text{id}"'] \arrow[rr, "{\Phi_{H_{\varphi},\overline{H_{\chi}}}}"] &  & {\text{Hom}_{\mathbbm{k}}((\overline{H_{\chi}})^{*}, H_{\varphi})} \arrow[d, "(\pi_{\varphi})_{*}"] \\
\overline{H_{\varphi}}\otimes \overline{H_{\chi}} \arrow[rr, "{\Phi_{\overline{H_{\varphi}},\overline{H_{\chi}}}}"]                        &  & {\text{Hom}_{\mathbbm{k}}((\overline{H_{\chi}})^{*}, \overline{H_{\varphi}})}                      
\end{tikzcd}$$
It suffices to prove \eqref{eq-clDechiu1}, which is divided in two steps. We show that $(\overline{H_{\chi}})^{*}\rightharpoonup J_{\varphi\ast\chi}$ is a right ideal of $H_{\varphi}$ first, then we prove that every element in this right ideal is nilpotent, which shows that \eqref{eq-clDechiu1} holds.

Now we show $(\overline{H_{\chi}})^{*}\rightharpoonup J_{\varphi\ast\chi}$ is a right ideal of $H_{\varphi}$.
Assume that $h,k\in H$ and $f\in (\overline{H_{\chi}})^{*}$. Then
\begin{align*}
&(f\rightharpoonup \pi_{\varphi\ast\chi}(h))\pi_{\varphi}(k)\\
=& \sum_{(f)}(f_{(1)}\rightharpoonup  \pi_{\varphi\ast \chi}(h))\pi_{\varphi}(k)\varepsilon_{(\overline{H_{\chi}})^{*}}(f_{(2)})\\
=& \sum_{(f)}(f_{(1)}\rightharpoonup  \pi_{\varphi\ast \chi}(h))(f_{(2)}\rightharpoonup (f_{(3)}\overline{S_{\chi^{-1}}}\rightharpoonup  \pi_{\varphi}(k))) & (\text{By Lemma}\ref{lem-Hpbarcoacthg2} (a)) \\
=&\sum_{(f)} f_{(1)}\rightharpoonup  (\pi_{\varphi\ast \chi}(h) (f_{(2)}\overline{S_{\chi^{-1}}}\rightharpoonup  \pi_{\varphi}(k)) ) & (\text{By Lemma}\ref{lem-Hpbarcoacthg2} (b)).
\end{align*}
Thus, if $\pi_{\varphi\ast\chi}(h)\in J_{\varphi\ast\chi}$, then
$$(f\rightharpoonup \pi_{\varphi\ast\chi}(h))\pi_{\varphi}(k)\in ((\overline{H_{\chi}})^{*}\rightharpoonup J_{\varphi\ast\chi} ),$$ for all $k\in H$. This shows that $(\overline{H_{\chi}})^{*}\rightharpoonup J_{\varphi\ast\chi}$ is a right ideal of $H_{\varphi}$.

Next, we show every element in  $(\overline{H_{\chi}})^{*}\rightharpoonup J_{\varphi\ast\chi}$ is nilpotent.  For any $x\in H_{\varphi}$, let $x_{l}$ denote the left-multiplication operator on $H_{\varphi}$ induced by $x$. Let $\text{tr}_{H_{\varphi}}:\text{End}_{\mathbbm{k}}(H_{\varphi})\to\mathbbm{k}$ be the usual trace map on the endomorphism algebra of $H_{\varphi}$.
By Cayley-Hamilton theorem, to prove that the right ideal $(\overline{H_{\chi}})^{*}\rightharpoonup J_{\varphi\ast\chi}$  is nil, it suffices to show that for any $f\in (\overline{H_{\chi}})^{*}$ and $x\in J_{\varphi\ast\chi}$,
\begin{equation}\label{eq-tractnil=0uJ}
\text{tr}_{H_{\varphi}}((f\rightharpoonup x)_{l})=0.
\end{equation}
In the discussion of the previous step, we obtained that for any $h\in H,f\in (\overline{H_{\chi}})^{*}$,
$$(f\rightharpoonup \pi_{\varphi\ast\chi}(h))_{l}=\sum_{(f)} f_{(1)}\rightharpoonup  (\pi_{\varphi\ast \chi}(h) (f_{(2)}\overline{S_{\chi^{-1}}}\rightharpoonup -) )\in \text{End}_{\mathbbm{k}}(H_{\varphi}).$$
It follows that
\begin{equation}\label{eq-profnilussetrid1}
\text{tr}_{H_{\varphi}}((f\rightharpoonup \pi_{\varphi\ast\chi}(h))_{l})=\sum_{(f)}\text{tr}_{H_{\varphi\ast\chi}}((f_{(2)}\overline{S_{\chi^{-1}}}\rightharpoonup -)\circ (f_{(1)}\rightharpoonup -)\circ(\pi_{\varphi\ast\chi}(h))_{l}).
\end{equation}
Since $\dim_{\mathbbm{k}}\overline{H_{\chi}}=\dim_{\mathbbm{k}}\overline{H_{\ovep}}$, $\overline{H_{\chi}}$ is an $\overline{H}_{\ovep}$-$\overline{H}_{\ovep}$-bi-Galois object, by Proposition \ref{prop-quobyJ-bgaloiobj-Chevs} (a). It follows from Proposition \ref{prop-quobyJ-bgaloiobj-Chevs} (b) that there exists a character $\theta\in G(H^{\circ})$ such that $\text{Ker }\theta\supseteq \m_{\ovep}$ (thus  $\theta $ can be regarded as an element of $(\overline{H_{\ovep}})^{*}$) and
\begin{equation}\label{eq-3st2uwindthete}
\pi_{\chi}(S^{2}(h))= \sum_{(h)}\theta(h_{(1)})\pi_{\chi}(h_{(2)})\in H_{\chi}.
\end{equation}
For any $k\in H$ and $g\in (\overline{H_{\chi}})^{*}$, we compute
\begin{align*}
g\rightharpoonup \pi_{\varphi\ast\chi}(k)=& \sum_{(k)}\pi_{\varphi}(k_{(1)})g(\pi_{\chi}(k_{(2)}))\\
=&  \sum_{(k)}\pi_{\varphi}(k_{(1)})g\overline{S_{\chi^{-1}}}\overline{S_{\chi}}(\overline{\pi_{\chi}(S^{2}(k_{(2)})}))\\
=&\sum_{(k)}\pi_{\varphi}(k_{(1)})g\overline{S_{\chi^{-1}}}\overline{S_{\chi}}(\theta(k_{(2)})\overline{\pi_{\chi}(k_{(3)})}) & (\text{By\ }\eqref{eq-3st2uwindthete})\\
=&  g\overline{S_{\chi^{-1}}}\overline{S_{\chi}}\rightharpoonup (\theta\rightharpoonup  \pi_{\varphi\ast\chi}(k)).
\end{align*}
Moreover, a direct computation shows that
\begin{equation}\label{eq-uselemore27}
\sum_{(f)}f_{(2)}\overline{S_{\chi^{-1}}}\rightharpoonup (f_{(1)}\overline{S_{\chi^{-1}}}\overline{S_{\chi}}\rightharpoonup  \pi_{\varphi\ast\chi}(h))=\pi_{\varphi\ast\chi}(h)\varepsilon_{(\overline{H_{\chi}})^{*}}(f).
\end{equation}
Combining with \eqref{eq-profnilussetrid1} and \eqref{eq-uselemore27}, we obtain that
\begin{equation}\label{eq-profnilussetrid2}
\text{tr}_{H_{\varphi}}((f\rightharpoonup \pi_{\varphi\ast\chi}(h))_{l})=\varepsilon_{(\overline{H_{\chi}})^{*}}(f)\text{tr}_{H_{\varphi\ast\chi}}((\theta\rightharpoonup -)\circ (\pi_{\varphi\ast\chi}(h))_{l}).
\end{equation}
By \eqref{eq-desbyjac-fibalgs1}, $\Delta_{\varphi\ast\chi}(J_{\varphi\ast \chi})\subseteq J_{\varphi\ast\chi}\otimes H_{\ovep}+H_{\varphi\ast\chi}\otimes J_{\ovep}$. Hence, it follows from $\theta\in (\overline{H_{\ovep}})^{*}$ that
$$((\theta\rightharpoonup -)\circ (J_{\varphi\ast\chi})_{l})(H_{\varphi\ast\chi})\subseteq J_{\varphi\ast\chi}.$$
By induction, one has $((\theta\rightharpoonup -)\circ (J_{\varphi\ast\chi})_{l})^{n}(H_{\varphi\ast\chi})\subseteq J_{\varphi\ast\chi}^{n}$ for all positive integer $n$. This shows that $(\theta\rightharpoonup -)\circ (J_{\varphi\ast\chi})_{l}\in \text{End}_{\mathbbm{k}}H_{\varphi\ast\chi}$ is nilpotent. Now it follows from \eqref{eq-profnilussetrid2} that $\text{tr}_{H_{\varphi}}((f\rightharpoonup x)_{l})=0$ for all $x\in J_{\varphi\ast\chi}$.

This completes the proof of \eqref{eq-clDechiu1}. Hence $\Delta_{\varphi,\chi}(J_{\varphi\ast\chi})\subseteq J_{\varphi}\otimes H_{\chi}+H_{\varphi}\otimes J_{\chi}$, provided $\dim_{\mathbbm{k}}\overline{H_{\chi}}=\dim_{\mathbbm{k}}\overline{H_{\ovep}}$. The case $\dim_{\mathbbm{k}}\overline{H_{\varphi}}=\dim_{\mathbbm{k}}\overline{H_{\ovep}}$ follows by a symmetric argument.

(b) For any irreducible $H_{\varphi}$-module $V$ and any irreducible $H_{\chi}$-module $W$, by (a), one has
$$J_{\varphi\ast\chi}(V\otimes W)=0.$$
It follows that $V\otimes W$ is a completely reducible $H_{\varphi\ast\chi}$-module.
\end{proof}

\subsection{Tensor-reducible representations}\label{sub-tensred}
Let $H$ be a Hopf algebra over $\mathbbm{k}$. An $H$-module $V$ is said to be \textit{left tensor-reducible}, if $V\otimes W$ is completely reducible for all irreducible $H$-modules $W$. Similarly, one can define the notion of \textit{a right tensor-reducible module}. An $H$-module is said to be \textit{tensor-reducible}, if it is both left tensor-reducible and right tensor-reducible. Since the trivial module $_{H}\mathbbm{k}$ satisfies that $W \cong W\otimes \mathbbm{k}\cong \mathbbm{k}\otimes W$ for all $H$-modules $W$, it is clear that any left or right tensor-reducible module is completely reducible.  It is obvious that if both $V$ and $W$ are tensor-reducible, $V\otimes W$ are also tensor-reducible. Note that an $H$-module is tensor-reducible if and only if it is a direct sum of some tensor-reducible irreducible $H$-modules. Hence, studying tensor-reducible representations of $H$ reduces to studying the tensor-reducible irreducible representations of $H$.

Let $(H,C)$ be a pair of Hopf algebras satisfying the hypothesis (FinHopf). Then all irreducible $H$-modules are finite-dimensional over $\mathbbm{k}$. By \cite[Theorem 3.1(b)]{mi2025lowest}, all $1$-dimensional $H$-modules are tensor-reducible. It is clear that the following are equivalent.
\begin{itemize}
\item[(a)] $H$ has the Chevalley property;
\item[(b)] All irreducible $H$-modules of $H$ are left tensor-reducible;
\item[(c)] All irreducible $H$-modules of $H$ are right tensor-reducible;
\item[(d)] All irreducible $H$-modules of $H$ are tensor-reducible.
\end{itemize}
Thus, the study of tensor-reducible irreducible representations of $H$ is closely related to the investigation of the Chevalley property of $H$, and the greater the density of tensor-reducible members among the irreducible representations of $H$, the closer $H$ is to possessing the Chevalley property.

\begin{example}\label{eg-grpalg-extenc}
Let $G$ be a finitely generated group with a central subgroup $N$ of index $m$. Clearly, $G$ is a central extension of the finite group $G/N$ by a finitely generated abelian group $N$. Moreover, $\mathbbm{k}G $ is a free $\mathbbm{k}N$-module of rank $m$. By Example \ref{eg-reutr1}, $(\mathbbm{k}G,\mathbbm{k}N,\text{tr}_{\text{reg}})$ is an affine Cayley-Hamilton Hopf algebra of degree $m$ and the identity fiber algebra is the group algebra $\mathbbm{k}[G/N]$. It is well-known that $\mathbbm{k}G$ has the Chevalley property, see for example \cite[Theorem 4.12.1]{MR3242743}. In this case, all irreducible $\mathbbm{k}G$-modules are tensor-reducible.
\end{example}
The following result shows that, under the assumption that the identity fiber algebra $H/\mathfrak{m}_{\ovep}H$ has the Chevalley property, if a fiber algebra $H/\mathfrak{m}H$ of $H$ admits an irreducible representation that is tensor-reducible as an $H$-module, then all irreducible $H/\mathfrak{m}H$-modules are tensor-reducible as $H$-modules.
\begin{proposition}
Let $(H,C)$ be a pair of Hopf algebras satisfying the hypothesis (FinHopf) such that the identity fiber algebra $H/\mathfrak{m}_{\ovep}H$ has the Chevalley property, and let $\mathfrak{m}\in \operatorname{maxSpec}C$. Then the following are equivalent.\label{prop-tensred-fibalg}
\begin{itemize}
\item[(a)] There exists an irreducible $H/\mathfrak{m}H$-module that is tensor-reducible as an $H$-module.
\item[(b)] All irreducible $H/\mathfrak{m}H$-modules are tensor-reducible as $H$-modules.
\end{itemize}
\end{proposition}
\begin{proof}
It suffices to prove that (a) implies (b).
Assume that there is an irreducible $H/\mathfrak{m}H$-module $V$ that is tensor-reducible as an $H$-module.
Then for any irreducible $H/\mathfrak{m}H$-module $W$, by \cite[Theorem 3.1 (a)]{mi2025lowest}, there is an irreducible $H/\mathfrak{m}_{\ovep}H$-module $M^{\prime}$ such that $W$ is a quotient of $M^{\prime}\otimes V$.
By Proposition \ref{thm-DelJinone-bardimeq1} (b), any irreducible $H/\mathfrak{m}_{\ovep}H$-module is tensor-reducible. In particular, $M^{\prime}$ is tensor-reducible.
It is clear that the tensor product of two tensor-reducible modules, $M^{\prime}\otimes V$, is itself tensor-reducible.
As $M^{\prime}\otimes V$ is completely reducible, $W$ is necessarily a direct summand of this module.
We therefore conclude that $W$ is tensor-reducible as an $H$-module.
\end{proof}


\subsection{Proof of Theorem \ref{mainthm-tensred4}}\label{subsec-pfofthm1}
For an algebra with trace $(A,C,\text{tr})$ such that $A$ is a finitely generated free $C$-module, the discriminant \cite{MR1972204} of $A$ over $C$ with respect to the trace map $\text{tr}$ is defined by
$$
\Delta(A/C;\text{tr})\coloneqq \text{det}(\text{tr}(x_{i}x_{j}))_{i,j=1}^{n}\in C,
$$
where $\{x_{1},...,x_{n}\}$ is a $C$-basis of $A$. The discriminant $\Delta(A/C;\text{tr})$ is unique up to multiplication by a unit in $C$. For example, the discriminant of an algebraic number field $L$ is precisely the discriminant of its ring of integers over $\mathbb{Z}$ with respect to the regular trace (Example \ref{eg-reutr1}). As the free condition of $A$ over $C$ is not always satisfied, discriminant ideals \cite{MR1972204} and their modified versions \cite{MR3415697} are used instead.

\begin{definition}
Let $(A,C,\text{tr})$ be an algebra with trace. For any positive integer $k$,\label{def-discriid}
\begin{itemize}
\item[(1)]the \textit{$k$-discriminant ideal} $D_{k}(A/C;\text{tr})$ is the ideal of $C$ generated by
$$\{\text{det}(\text{tr}(a_{i}a_{j}))_{i,j=1}^{k}\mid (a_{1},...,a_{k})\in A^{k}\}.$$
\item[(2)]the \textit{modified $k$-discriminant ideal} $MD_{k}(A/C;\text{tr})$ is the ideal of $C$ generated by
$$\{\text{det}(\text{tr}(a_{i}b_{j}))_{i,j=1}^{k}\mid (a_{1},...,a_{k}),(b_{1},...,b_{k})\in A^{k}\}.$$
\end{itemize}
\end{definition}

Let $(A,C,\text{tr})$ be an affine Cayley-Hamilton algebra. Then $A$ is a module-finite $C$-algebra and $C$ is an affine $\mathbbm{k}$-algebra by \cite[Theorem 4.5 (a)]{MR1288995}. It is convenient to work with the \textit{square dimension function} \cite{mi2025lowest}:
\begin{equation}\label{eq-Sd-function}
\text{Sd}:\operatorname{maxSpec}C\to \mathbb{N}, \mathfrak{m}\mapsto \sum_{[V]\in \text{Irr}(A/\mathfrak{m}A)}(\dim_{\mathbbm{k}}V)^{2}.
\end{equation}
By \text{\cite[Theorem 4.1 (b)]{MR3886192}}, for any positive integer $k$, one has
\begin{equation}\label{eq-thm-BY-disrep}
\begin{aligned}
\mathcal{V}_{k}\coloneqq \mathcal{V}(D_{k}(A/C;\text{tr}))=&\mathcal{V}(MD_{k}(A/C;\text{tr}))\\
=& \{\mathfrak{m}\in \operatorname{maxSpec}C\mid\text{Sd}(\mathfrak{m})<k \}.
\end{aligned}
\end{equation}

Let us now restrict to the Hopf context. For an affine Cayley-Hamilton Hopf algebra $(H,C,\text{tr})$, by Cartier’s theorem, $C$ is reduced and $\operatorname{maxSpec}C$ is an affine algebraic group. Moreover, for evey positive integer k, the $k$-discriminant subvariety $\mathcal{V}_{k}$ is a closed subvariety of the algebraic group $\operatorname{maxSpec}C$.  In \cite[Theorem 4.1]{HMQWChev2025}, the authors proved that if the identity fiber algebra $H/\mathfrak{m}_{\ovep}H$ has the Chevalley property, then the lowest discriminant ideal $D_{\ell}(H/C;\text{tr})$ of $(H,C,\text{tr})$ is of level
\begin{equation}\label{eq-CHHopf-lowsd}
\ell\coloneqq \text{Sd}(\mathfrak{m}_{\ovep})+1=\text{FPdim}(\text{Gr}(H/\mathfrak{m}_{\ovep}H))+1,
\end{equation}
where $\text{FPdim}(\text{Gr}(H/\mathfrak{m}_{\ovep}H))$ is the Frobenius-Perron dimension of the Grothendieck ring $\text{Gr}(H/\mathfrak{m}_{\ovep}H)$. The following lemma follows immediately from \eqref{eq-thm-BY-disrep} and \eqref{eq-CHHopf-lowsd}.
\begin{lemma}
Let $(H,C,\text{tr})$ be an affine Cayley-Hamilton Hopf algebra such that the identity fiber algebra $H/\mathfrak{m}_{\ovep}H$ has the Chevalley property, let $\ell$ be the positive integer given in \eqref{eq-CHHopf-lowsd}. Then, for any $\mathfrak{m}\in \operatorname{maxSpec}C$, $\mathfrak{m} \supseteq D_{\ell}(H/C;\text{tr}) $ if and only if $\text{Sd}(\mathfrak{m})=\text{Sd}(\mathfrak{m}_{\ovep})$.\label{lem-sdeqiffinlows3}
\end{lemma}

Let $G_{0}\coloneqq G((H/\mathfrak{m}_{\ovep}H)^{\circ})$ be the group of characters of the identity fiber algebra $H/\mathfrak{m}_{\ovep}H$. Then for any $\mathfrak{m}\in \operatorname{maxSpec}C$, the set $\text{Irr}(H/\mathfrak{m}H)$ admits a natural $G_{0}$-action via \eqref{eq-bifunct-conv-fiberalgsten}:
$$G_{0}\times \text{Irr}(H/\mathfrak{m}H)\to \text{Irr}(H/\mathfrak{m}H),  \  (\chi, [V])\mapsto [\chi\otimes V].$$
For any $[V]\in \text{Irr}(H/\mathfrak{m}H)$, let $\text{Stab}_{G_{0}}[V]$ be the stabilizer of $[V]$ under the above $G_{0}$-action. It always holds that $\text{Stab}_{G_{0}}[V]\leq (\dim_{\mathbbm{k}}V)^{2}$ \cite[Proposition 3.5 (b)]{mi2025lowest}. When the equality holds, $V$ is said to be maximally stable \cite[Definition 3.8]{mi2025lowest}.
\begin{theorem}
Let $(H,C,\text{tr})$ be an affine Cayley-Hamilton Hopf algebra such that the identity fiber algebra $H/\mathfrak{m}_{\ovep}H$ has the Chevalley property.  Let $\ell$ be the positive integer given in \eqref{eq-CHHopf-lowsd}. Then
\begin{itemize}
\item[(a)] for any irreducible $H$-module $V$, the following are equivalent.\label{thm-tensred-insec2}
\begin{itemize}
\item[(i)] $V$ is tensor-reducible;
\item[(ii)] $V$ is left tensor-reducible;
\item[(iii)] $V$ is right tensor-reducible;
\item[(iv)] $V$ is annihilated by the lowest discriminant ideal  $D_{\ell}(H/C;\text{tr})$;
\item[(v)] $V$ is annihilated by the lowest modified discriminant ideal $MD_{\ell}(H/C;\text{tr})$;
\item[(vi)] $V\otimes V^{*}$ is completely reducible.
\end{itemize}
\item[(b)] all maximally stable irreducible $H$-modules are tensor-reducible. The converse holds provided that $H/\mathfrak{m}_{\ovep}H$ is a basic algebra.
\end{itemize}
\end{theorem}
\begin{proof}
(a) (i) $\Rightarrow$ (ii), (i) $\Rightarrow$ (iii) and (i) $\Rightarrow$ (vi) are trivial.  

(iv)  $\Leftrightarrow$ (v) follows immediately from \eqref{eq-thm-BY-disrep} and the fact that $\text{Ann}_{C}V \in \operatorname{maxSpec} C$.

(iv) $\Rightarrow$ (i) Suppose that $V$ is annihilated by $D_{\ell}(H/C;\text{tr})$. Set $\mathfrak{m}\coloneqq \text{Ann}_{C}V\in \operatorname{maxSpec}C$. Then $V$ can be regarded as an irreducible $H/\mathfrak{m}H$-module. By Lemma \ref{lem-sdeqiffinlows3}, $\text{Sd}(\mathfrak{m})=\text{Sd}(\mathfrak{m}_{\ovep})$. For any irreducible $H$-module $W$, set $\mathfrak{n}\coloneqq \text{Ann}_{C}W\in  \operatorname{maxSpec}C$. By \eqref{eq-bifunct-conv-fiberalgsten}, $V\otimes W$ is an $H/(\mathfrak{m}\ast\mathfrak{n})H$-module, and $W\otimes V$ is an $H/(\mathfrak{n}\ast\mathfrak{m})H$-module. By Proposition \ref{thm-DelJinone-bardimeq1} (b), both $V\otimes W$ and $W\otimes V$ are completely reducible as $H$-modules. Hence $V$ is tensor-reducible.

 (ii) $\Rightarrow$ (iv) By Lemma \ref{lem-sdeqiffinlows3}, it suffices to show that $\mathfrak{m}\coloneqq\text{Ann}_{C}V$ satisfies $\text{Sd}(\mathfrak{m})=\text{Sd}(\mathfrak{m}_{\ovep})$. For each $\mathfrak{t}\in \operatorname{maxSpec}C$, set $\overline{H_{\mathfrak{t}}}\coloneqq (H/\mathfrak{t}H)/\text{Jac}(H/\mathfrak{t}H)$, which is completely reducible as an $H$-module. 

Set $\text{Irr}(H/\mathfrak{m}_{\ovep}H)=\{[M_{1}],...,[M_{d}]\}$. Notice that $D_{\ell}(H/C;\text{tr})M_{j}=0$ for any $1\leq j\leq d$. Hence $M_{j}$ is tensor-reducible for any $1\leq j\leq d$. Consider the adjoint isomorphism
$$\text{Hom}_{H}(\overline{H_{\mathfrak{m}^{-1}}}\otimes V,M_{j})\cong  \text{Hom}_{H}(\overline{H_{\mathfrak{m}^{-1}}} ,M_{j}\otimes V^{*}),$$
where $V^{*}$ is the left dual of $V$. Note that $V^{*}$ is irreducible and $V$ is left tensor-reducible,
\begin{align*}
[\overline{H_{\mathfrak{m}^{-1}}}\otimes V:M_{j}]=&\dim_{\mathbbm{k}}\text{Hom}_{H}(\overline{H_{\mathfrak{m}^{-1}}}\otimes V,M_{j})\\
=&\dim_{\mathbbm{k}}\text{Hom}_{H}(\overline{H_{\mathfrak{m}^{-1}}} ,M_{j}\otimes V^{*})\\
=&\dim_{\mathbbm{k}}\text{Hom}_{H/\mathfrak{m}^{-1}H}(\overline{H_{\mathfrak{m}^{-1}}} ,M_{j}\otimes V^{*})\\
=&(\dim_{\mathbbm{k}}M_{j})(\dim_{\mathbbm{k}}V).
\end{align*}
The last equality follows from the fact that $M_{j}\otimes V^{*}$ is completely reducible. Hence $$\overline{H_{\mathfrak{m}^{-1}}}\otimes V\cong \overline{H_{\mathfrak{m}_{\ovep}}}^{\oplus \dim_{\mathbbm{k}}V}$$ as $H$-modules. It follows that 
$\text{Sd}(\mathfrak{m}^{-1})\dim_{\mathbbm{k}}V=\text{Sd}(\mathfrak{m}_{\ovep})\dim_{\mathbbm{k}}V.$ Hence $\text{Sd}(\mathfrak{m})=\text{Sd}(\mathfrak{m}_{\ovep})$. Therefore (iv) holds.

 (iii) $\Rightarrow$ (iv) It is similar to (ii) $\Rightarrow$ (iv).

 (vi) $\Rightarrow$ (v) Assume that $V\otimes V^{*}$ is completely reducible. Set $\mathfrak{m}\coloneqq \text{Ann}_{C}V$. By \cite[Theorem 4.1 (2)]{HMQWChev2025}, it suffices to show that $L\otimes L^{*}$ is completely reducible for any irreducible $H/\mathfrak{m}H$-module $L$. Let $L$ be an irreducible $H/\mathfrak{m}H$-module. It follows from \cite[Theorem 3.1 (a)]{mi2025lowest} that there exists an irreducible $H/\mathfrak{m}_{\ovep}H$-module $W$ such that $L$ is a submodule of $W\otimes V$. Since $V\otimes V^{*}$ is a direct sum of irreducible $H/\mathfrak{m}_{\ovep}H$-modules, $V\otimes V^{*}$ is tensor-reducible. Hence $W\otimes (V\otimes V^{*})\otimes W^{*}\cong (W\otimes V)\otimes (W\otimes V)^{*}$ is completely reducible. Note that $L\otimes L^{*}$ is a subquotient of $(W\otimes V)\otimes (W\otimes V)^{*}$. It follows that $L\otimes L^{*}$ is completely reducible.
 
Hence, the proof of (a) is completed.

(b) By \cite[Proposition 3.10]{mi2025lowest}, for any maximally stable irreducible $H$-module $V$, $V\otimes V^{*}$ is a direct sum of nonisomorphic $1$-dimensional $H$-module. It follows from (a) that $V$ is tensor-reducible.

Assume further that the identity fiber algebra $H/\mathfrak{m}_{\ovep}H$ is basic. By \cite[Theorem 4.2 (c)]{mi2025lowest},  any irreducible $H$-module annihilated by $D_{\ell}(H/C;\text{tr})$ is maximally stable. It follows from (a) that
any tensor-reducible irreducible $H$-module are maximally stable.
\end{proof}


\begin{remark}
In the setting of Theorem \ref{thm-tensred-insec2}, tensor-reducible irreducible $H$-module may not be maximally stable, even under the assumption that the identity fiber algebra $H/\mathfrak{m}_{\ovep}H$ is semisimple. See \cite[Remark 4.2]{HMQWChev2025}.
\end{remark}

Recall that the left winding automorphism $W_{l}(\chi)$ and right winding automorphism  $W_{r}(\chi)$ associated to $\chi\in G(H^{\circ})$ are defined respectively by:
$$W_{l}(\chi)(h)\coloneqq \sum_{(h)}\chi(h_{(1)})h_{(2)} \text{ and\ } W_{r}(\chi)(h)\coloneqq \sum_{(h)}h_{(1)}\chi(h_{(2)}), \forall h\in H.$$
It is obvious that the central Hopf subalgebra $C$ is invariant under the action of $W_{l}(\chi)$ and $W_{r}(\chi)$. Both $W_{l}(G(H^{\circ}))$ and $W_{r}(G(H^{\circ}))$ are subgroups of the algebra automorphism group of $H$, which are called the left winding automorphism group and the right winding automorphism group of $H$, respectively. Winding automorphisms play an important role in the study of the homological properties of $H$; for relevant material we refer to \cite{MR2437632}.

\begin{theorem}
Let $(H,C,\text{tr})$ be an affine Cayley-Hamilton Hopf algebra such that the identity fiber algebra $H/\mathfrak{m}_{\ovep}H$ has the Chevalley property. The following are equivalent.\label{thm-whenalltriepare1d}
\begin{itemize}
\item[(i)] All tensor-reducible irreducible $H$-modules are $1$-dimensional;
\item[(ii)] The identity fiber algebra $H/\mathfrak{m}_{\ovep}H$ is basic and all maximally stable irreducible $H$-modules are $1$-dimensional;
\item[(iii)] The identity fiber algebra $H/\mathfrak{m}_{\ovep}H$ is basic, and 
$$V_{\ell}=W_{l}(G(H^{\circ}))(\mathfrak{m}_{\ovep})=W_{r}(G(H^{\circ}))(\mathfrak{m}_{\ovep}),$$
where $\ell$ is the positive integer defined in \eqref{eq-CHHopf-lowsd}.
\end{itemize}
\end{theorem}
\begin{proof}
(i) $\Rightarrow $ (ii) By \cite[Theorem 4.1]{HMQWChev2025} and Theorem \ref{thm-tensred-insec2} (a), any irreducible $H/\mathfrak{m}_{\ovep}H$-module is tensor-reducible as an $H$-module. In particular, all irreducible $H/\mathfrak{m}_{\ovep}H$-modules are $1$-dimensional, which shows that the identity fiber algebra $H/\mathfrak{m}_{\ovep}H$ is basic. The second conclusion follows from Theorem \ref{thm-tensred-insec2} (b).

(ii) $\Rightarrow$ (iii) This is \cite[Theorem 4.3 (c)]{mi2025lowest}.

(iii) $\Rightarrow$ (i) By \cite[Theorem 4.3 (a)]{mi2025lowest}, for any $\mathfrak{m}\in \mathcal{V}_{\ell}$, $H/\mathfrak{m}H$ is basic. Hence the result follows from Theorem \ref{thm-tensred-insec2} (a). 
\end{proof}


\subsection{Big quantized Borel subalgebras at roots of unity}\label{subsec-Uqb}
Let $\mathfrak{g}$ be a finite-dimensional complex simple Lie algebra of rank $n$ with a Cartan subalgebra $\mathfrak{h}$, let $\Phi$ be the corresponding root system with a fixed base $ \Pi=\{\alpha_{1},...,\alpha_{n}\}$ and let $W$ be the Weyl group of $\Phi$. Let $(a_{ij})_{n\times n}$ denote the corresponding Cartan matrix of $\Phi$. Then there are coprime integers $d_{1},...,d_{n}\in \{1,2,3\}$ such that the matrix $(d_{i}a_{ij})_{n\times n}$ is symmetric. Let $l\geq 3$ be an odd positive integer, coprime to $3$ if $\mathfrak{g}$ is of type $G_{2}$, and let $\epsilon \in \mathbb{C}$ be a primitive $l$-th root of unity. We also assume that  $l$ is good for the root system $\Phi$, that is, if $\sum_{i=1}^{n}a_{i}\alpha_{i}$ is the highest root of $\mathfrak{g}$, then $l$ is prime to all the integers $a_{i}$. 

Throughout this subsection, $\mathbb{C}$ is the base field.


 Let $\mathcal{U}_{\epsilon}(\mathfrak{g})$ be the De Concini-Kac quantuzed enveloping algebra at the root of unity $\epsilon$ \cite{MR1103601}. This is a noncommutative and noncocommutative Hopf algebra with generators $E_{i},F_{i},K_{i}^{\pm 1}, 1\leq i\leq n$. The De Concini-Kac-Procesi quantized Borel subalgebra $\mathcal{U}_{\epsilon}^{\geq 0}(\mathfrak{g})$ \cite{MR1351503,MR1288995} is the subalgebra of $\mathcal{U}_{\epsilon}(\mathfrak{g})$ generated by $E_{i},K_{i}^{\pm 1}, 1\leq i\leq n$. $\mathcal{U}_{\epsilon}^{\geq 0}(\mathfrak{g})$  is a Hopf subalgebra of $\mathcal{U}_{\epsilon}(\mathfrak{g})$ with the coalgebra structure and antipode are given by 
$$\Delta(K_{i})=K_{i}\otimes K_{i},\  \Delta(E_{i})=E_{i}\otimes 1+K_{i}\otimes E_{i},$$
$$\varepsilon(K_{i})=1,\ \varepsilon(E_{i})=0,$$
$$S(K_{i})=K_{i}^{-1},\ S(E_{i})=-K_{i}^{-1} E_{i},$$
for $1\leq i\leq n$.






Let $w_{0}$ be the unique element of longest  length in the Weyl group $W$, and we fix a reduced expression $w=s_{i_{1}}s_{i_{2}}\cdots s_{i_{N}}$, where $s_{k}$ denote the simple reflection corresponding to each simple root $\alpha_{k}$. Define $\beta_{k}=s_{i_{1}}\cdots s_{i_{k-1}}(\alpha_{i_{k}}),k=1,2,...,N$. Then $\{\beta_{1},...,\beta_{N}\}$ is precisely the set of positive roots of $\Phi$ with respect to $\Pi$. Consider the quantum root vectors
$$E_{\beta_{k}}=T_{i_{1}}\cdots T_{i_{k-1}}(E_{i_{k}}),k=1,2,...,N,$$
where $T_{i}$ is the action \cite{MR1066560} of the braid group of $W$ on $\mathcal{U}_{\epsilon}(\mathfrak{g})$. 

By \cite[Proposition 1.7(c)]{MR1103601} or \cite[Theorem. III. 6.1]{MR1898492}, 
\begin{equation}\label{eq-Uqb-PBW}
\{K_{1}^{t_{1}}\cdots K_{n}^{t_{n}}E_{\beta_{1}}^{m_{1}}\cdots E_{\beta_{N}}^{m_{N}}\mid t_{1},...,t_{n}\in \mathbb{Z},m_{1},...,m_{N}\in \mathbb{N}\}
\end{equation}
 is a $\mathbb{C}$-basis of $\mathcal{U}_{\epsilon}^{\geq 0}(\mathfrak{g})$. 
 
 Consider the De Concini-Kac-Procesi \cite{MR1124981,MR1351503} central subalgebra
$$\mathcal{C}_{\epsilon}^{\geq 0}(\mathfrak{g})\coloneqq \mathbb{C}[K_{1}^{\pm l},...,K_{n}^{\pm l},E_{\beta_{1}}^{l},...,E_{\beta_{N}}^{l}]\subseteq \mathcal{U}_{\epsilon}^{\geq 0}(\mathfrak{g}) .$$
This is a central Hopf subalgebra of $\mathcal{U}_{\epsilon}^{\geq 0}(\mathfrak{g})$ by \cite[Proposition 5.6(d)]{MR1124981}. It is clear from \eqref{eq-Uqb-PBW} that $\mathcal{U}_{\epsilon}^{\geq 0}(\mathfrak{g})$ is a free $\mathcal{C}_{\epsilon}^{\geq 0}(\mathfrak{g})$-module of rank $l^{n+N}$. Hence, $(\mathcal{U}_{\epsilon}^{\geq 0}(\mathfrak{g}),\mathcal{C}_{\epsilon}^{\geq 0}(\mathfrak{g}),\text{tr}_{\text{reg}})$ is an affine Cayley-Hamilton Hopf algebra of degree $l^{n+N}$. By \cite[Lemma 5.3 (b)]{mi2025lowest} or \cite[Theorem 2.3 (a)]{MR1863398}, the identity fiber algebra $\mathcal{U}_{\epsilon}^{\geq 0}(\mathfrak{g})/\mathfrak{m}_{\ovep}\mathcal{U}_{\epsilon}^{\geq 0}(\mathfrak{g})\cong \mathfrak{u}_{\epsilon}^{\geq 0}(\mathfrak{g})$ is basic. In particular, it has the Chevalley property. This observation enables us to use Theorem \ref{thm-whenalltriepare1d} to show that the tensor-reducible irreducible representations of $\mathcal{U}_{\epsilon}^{\geq 0}(\mathfrak{g})$ are precisely the $1$-dimensional ones.

\begin{proposition}\label{prop-quaborel-rofutenr}
Let $\mathfrak{g}$ be a finite-dimensional complex simple Lie algebra of rank $n$, let $l\geq 3$ be an odd positive integer, coprime to $3$ if $\mathfrak{g}$ is of type $G_{2}$, and let $\epsilon\in \mathbb{C}$ be a primitive $l$-th root of unity. For any irreducible $\mathcal{U}_{\epsilon}^{\geq 0}(\mathfrak{g})$-module $V$, the following are equivalent:
\begin{itemize}
\item[(i)] $V$ is tensor-reducible;
\item[(ii)] $V$ is left tensor-reducible;
\item[(iii)] $V$ is right tensor-reducible;
\item[(iv)] $V$ is $1$-dimensional;
\item[(v)] $V$ is maximally stable.
\end{itemize}
\end{proposition}
\begin{proof}
Note that $(\mathcal{U}_{\epsilon}^{\geq 0}(\mathfrak{g}),\mathcal{C}_{\epsilon}^{\geq 0}(\mathfrak{g}),\text{tr}_{\text{reg}})$ is an affine Cayley-Hamilton Hopf algebra with basic identity fiber algebra. By \cite[Theorem 5.5 (b)]{mi2025lowest}, the lowest discriminant variety 
$$\mathcal{V}(D_{l^{n}+1}(\mathcal{U}_{\epsilon}^{\geq 0}(\mathfrak{g})/\mathcal{C}_{\epsilon}^{\geq 0}(\mathfrak{g});\text{tr}_{\text{reg}}))=W_{l}((\mathcal{U}_{\epsilon}^{\geq 0}(\mathfrak{g}))^{\circ})(\mathfrak{m}_{\ovep})=W_{r}((\mathcal{U}_{\epsilon}^{\geq 0}(\mathfrak{g}))^{\circ})(\mathfrak{m}_{\ovep}).$$
Hence the result follows from Theorems \ref{thm-tensred-insec2} and \ref{thm-whenalltriepare1d}.
\end{proof}
\begin{remark}
By using \cite[Theorem 5.7 (b)]{mi2025lowest}, the same argument shows that for any connected, simply connected complex simple algebraic group  $G$, all tensor-reducible irreducible representations of the De Concini-Lyubashenko quantized coordinate algebra $\mathcal{O}_{\epsilon}(G)$ \cite{MR1296515} at a primitive $l$-th root of unity $\epsilon$ (as before, some restrictions on $l$ are assumed) are $1$-dimensional.
\end{remark}

\section{Chevalley property of Hopf algebras}\label{sec-3}

\subsection{Chevalley property of Cayley-Hamilton Hopf algebras}
In this subsection, we undertake a detailed study of the Chevalley property for affine Hopf algebras that admit large central Hopf subalgebras after \cite{HMQWChev2025}. We begin by establishing the connection between the Chevalley property and the lowest discriminant ideals for Cayley-Hamilton Hopf algebras.

\begin{theorem}
Let $(H,C,\text{tr})$ be an affine Cayley-Hamilton Hopf algebra over $\mathbbm{k}$. Then the following are equivalent:\label{thm-Cheviff-ap}
\begin{itemize}
\item[(i)] $H$ has the Chevalley property;
\item[(ii)] The identity fiber algebra $H/\mathfrak{m}_{\ovep}H$ has the Chevalley property, and the lowest discriminant ideal $D_{\ell}(H/C;\text{tr})=0$;
\item[(iii)] The identity fiber algebra $H/\mathfrak{m}_{\ovep}H$ has the Chevalley property, and the lowest modified discriminant ideal $MD_{\ell}(H/C;\text{tr})=0$;
\item[(iv)]The identity fiber algebra $H/\mathfrak{m}_{\ovep}H$ has the Chevalley property, and all the discriminant ideals of $(H,C,\text{tr})$ are trivial;
\item[(v)]The identity fiber algebra $H/\mathfrak{m}_{\ovep}H$ has the Chevalley property, and all the modified discriminant ideals of $(H,C,\text{tr})$ are trivial.
\end{itemize}
\end{theorem}
\begin{proof}
Note that $C$ is reduced by Cartier's Theorem. Hence the equivalence (ii) $\Leftrightarrow$ (iii) $\Leftrightarrow$ (iv) $\Leftrightarrow$ (v) follows immediately from \eqref{eq-thm-BY-disrep}. The implication (i) $\Rightarrow$ (iv) follows from \cite[Theorem 4.6]{HMQWChev2025}. Thus it suffices to prove (ii) $\Rightarrow$ (i). Assume that the lowest discriminat ideal $D_{\ell}(H/C;\text{tr})=0$. Then all irreducible $H$-modules are tensor-reducible by  Theorem \ref{thm-tensred-insec2} (a). If follows that $H$ has the Chevalley property.
\end{proof}

There exist affine Cayley-Hamilton Hopf algebras all of whose discriminant ideals are trivial, but they do not have the Chevalley property. For instance, pick any finite-dimensional Hopf algebra $H$ without the Chevalley property over $\mathbbm{k}$ and set $C=\mathbbm{k}$.

For any affine prime Hopf algebra $H$ that admits a large central Hopf subalgebra $C$, it was shown in \cite[Corollary 4.10]{HMQWChev2025} that $H$ has the Chevalley property if and only if $H$ is commutative. However, this is not the case when $H$ is only assumed to be semiprime \cite[Remark 4.11]{HMQWChev2025}. Below is an example of an affine Cayley-Hamilton Hopf algebra that has the Chevalley property but is not semiprime.



\begin{example}
Let $n\geq 2$ be a positive integer and let $\epsilon$ be a primitive $n$-th root of unity.\label{eg-aCheegwith1dimz-u2} Consider the Hopf algebra $H$ generated by $x,g^{\pm 1}$ with relations $gg^{-1}=g^{-1}g=1,xg=\epsilon gx$ and $x^{n}=0$. There is a unique Hopf algebra structure on $H$ such that
$$\Delta(g)=g\otimes g,\ \varepsilon(g)=1,\ \Delta(x)=x\otimes g+1\otimes x,\ \varepsilon(x)=0.$$
Clearly, $H$ is not semiprime. Note that $C=\mathbbm{k}[g^{\pm n}]$ is a central Hopf subalgebra of $H$ and $H$ is a free $C$-module of rank $n^{2}$. It follows that $(H,C,\text{tr}_{\text{reg}})$ is an affine Cayley-Hamilton Hopf algebra of degree $n^{2}$. It is clear that the identity fiber algebra $H/\mathfrak{m}_{\ovep}H=\mathbbm{k}\langle x,g\rangle/(g^{n}-1,x^{n},xg-\epsilon gx)$ is the $n^{2}$-dimensional Taft algebra, and hence basic. Moreover, it is not hard to see that all fiber algebras admit $1$-dimensional representations. By \cite[Theorem 4.3 (a)(b)]{mi2025lowest}, the lowest discriminant subvariety is the entire maximal spectrum $\operatorname{maxSpec}C$. It is clear  that $H$ has the Chevalley property.
\end{example}

In \cite[Corollary 4.10]{HMQWChev2025}， the authors also proved that a prime affine Hopf algebra $H$ admitting a large central Hopf subalgebra $C$ has the Chevalley property if and only if its identity fiber algebra $H/\mathfrak{m}_{\ovep}H$ is semisimple. It is worth noting that the same statement remains valid in the semiprime setting.

\begin{proposition}
Let $(H,C)$ be a pair of Hopf algebras satisfying the hypothesis (FinHopf). Then the following are equivalent.\label{prop-semiprimeChev}
\begin{itemize}
\item[(a)] $H$ is regular (i.e., the global dimension of $H$ is finite) and has the Chevalley property;
\item[(b)] $H$ is semiprime and has the Chevalley property;
\item[(c)] The identity fiber algebra $H/\mathfrak{m}_{\ovep}H$ is semisimple;
\item[(d)] All identity fiber algebras $H/\mathfrak{m}H$ are semisimple.
\end{itemize}
\end{proposition}
\begin{remark}
For any semiprime affine Hopf algebra $H$ admitting large central Hopf subalgebra $C$, $H$ is commutative if and only if all irreducible representations of $H$ are $1$-dimensional (the semiprime condition is necessary by Example \ref{eg-aCheegwith1dimz-u2}). In fact, if $H$ is semiprime and all its irreducible representations are $1$-dimensional, then all the fiber algebras $H/\mathfrak{m}H$ are commutative (since they are both basic and semisimple). Therefore, for any $a,b\in H$, one has
$$ab-ba\in \bigcap_{\mathfrak{m}\in \operatorname{maxSpec}C}\mathfrak{m}H\subseteq \text{Jac}(H)=0.$$
The last equality follows from the fact that $H$ is semiprime and $H$ is a Jacobson ring \cite[Theorem 13.10.3 (iii)]{MR1811901}. This shows that $H$ is commutative. The conclusion in the other direction is clear. 
\end{remark}
We will need the following.
\begin{lemma}
Let $(H,C)$ be a pair of Hopf algebras satisfying the hypothesis (FinHopf). If $H$ is semiprime, then there is $\mathfrak{m}_{0}\in \operatorname{maxSpec}C$ such that the fiber algebra $H/\mathfrak{m}_{0}H$ is semisimple.\label{lem-semipri-ssfa}
\end{lemma}
\begin{proof}
Let $H^{e}\coloneqq H\otimes_{C}H^{op}$, $K$ be the kernel of the multiplication map $\mu:H^{e}\to H, x\otimes y\mapsto xy$. Then $0 \to  K \to  H^{e} \xrightarrow{\mu}  H \to 0$ is a short exact sequence of left $H^{e}$-modules.

We claim that there is some $\mathfrak{m}_{0}\in \operatorname{maxSpec}C$ such that $\text{Ext}_{(H_{\mathfrak{m}_{0}})^{e}}^{1}(H_{\mathfrak{m}_{0}},K_{\mathfrak{m}_{0}})=0$. Once the claim is established, $H_{\mathfrak{m}_{0}}$ is a separable $C_{\mathfrak{m}_{0}}$-algebra. So, $H/\mathfrak{m}_{0}H$ is separable over $C/\mathfrak{m}_{0}\cong \mathbbm{k}$ \cite[Theorem 7.1]{MR280479}. In particular, $H/\mathfrak{m}_{0}H$ is semisimple. Now we proceed to prove the aforementioned claim. Assume, by way of contradiction, that $\text{Ext}_{(H_{\mathfrak{m}})^{e}}^{1}(H_{\mathfrak{m}},K_{\mathfrak{m}})\neq 0$ for all $\mathfrak{m}\in \operatorname{maxSpec}C$. Note that $H^{e}$ is noetherian. Then
$$(\text{Ext}_{H^{e}}^{1}(H,K))_{\mathfrak{m}}\cong \text{Ext}_{(H_{\mathfrak{m}})^{e}}^{1}(H_{\mathfrak{m}},K_{\mathfrak{m}})\neq 0$$
for all $\mathfrak{m}\in \operatorname{maxSpec}C$. Since $C$ is reduced, the annihilator $\text{Ann}_{C}(\text{Ext}_{H^{e}}^{1}(H,K))=0$. Let $T$ be the multiplicative subset of $C$ consisting of all regular elements. Then $C_{T}$ is semisimple artinian. It follows that $H_{T}$ is also semisimple artinian. Since the base field we work over is of characteristic zero, $H_{T}$ is separable over $C_{T}$. Hence $(\text{Ext}_{H^{e}}^{1}(H,K))_{T}=0$, which implies that there is some $t\in T$ such that $t\in \text{Ann}_{C}(\text{Ext}_{H^{e}}^{1}(H,K))=0$. Contradiction!
\end{proof}

\begin{proof}[Proof of Proposition \upshape\ref{prop-semiprimeChev}]
The implication (a) $\Rightarrow$ (b) follows from \cite[Proposition. III. 8.2]{MR1898492} and \cite[Corollary 4.2]{MR719665}. The implication (c) $\Rightarrow$ (a) is \cite[Theorem 2.9 (8)]{MR4201485}. The equivalences (b) $\Leftrightarrow$ (c) $\Leftrightarrow$ (d) follows immediately from the proof of \cite[Corollary 4.10]{HMQWChev2025} and Lemma \ref{lem-semipri-ssfa}.
\end{proof}

The following proposition enables us to construct numerous examples of affine Hopf algebras which admit large central Hopf subalgebras and have the Chevalley property.
\begin{proposition}\label{prop-tensorChevisChev}
Let $(H,C)$ and $(\widetilde{H},\widetilde{C})$ be two pairs satisfying the hypothesis (FinHopf). Then $(H\otimes \widetilde{H},C\otimes \widetilde{C})$ satisfies the hypothesis (FinHopf) as well. If, in addition,  both $H$ and $\widetilde{H}$ have the Chevalley property, then so does $H\otimes \widetilde{H}$.
\end{proposition}
\begin{proof}
The first statement is clear. We proceed to prove the second. Assume that both $H$ and $\widetilde{H}$ have the Chevalley property. Note that the irreducible modules of the three Hopf algebras $H,\widetilde{H} $ and $H\otimes \widetilde{H}$ are all finite-dimensional over $\mathbbm{k}$ \cite[Theorem 13.10.3(i)]{MR1811901}. By \cite[Theorem 3.10.2]{MR2808160}, for any irreducible $H$-module $X$ and any irreducible $\widetilde{H}$-module $\widetilde{X}$, the tensor product $X\otimes \widetilde{X}$ is an irreducible $H\otimes \widetilde{H}$-module, and evey irreducible $H\otimes \widetilde{H}$-module is of this form up to isomorphism. Moreover, for any $H$-modules $X,Y$ and any $\widetilde{H}$-modules $\widetilde{X},\widetilde{Y}$, one has a canonical isomorphism
$$ (X\otimes Y)\otimes (\widetilde{X}\otimes \widetilde{Y})\cong  (X\otimes \widetilde{X}) \otimes  (Y\otimes \widetilde{Y})$$
as $H\otimes \widetilde{H}$-modules. It follows that $H\otimes \widetilde{H}$ has the Chevalley property.
\end{proof}

For any positive integer $d$ and any finite group $G$, the tensor product algebra $B=\mathbbm{k}[x_{1},...,x_{d}]\otimes \mathbbm{k}G$ is module-finite over the central Hopf subalgebra $\mathbbm{k}[x_{1},...,x_{d}]\otimes 1$ (the Hopf algebra structure on $\mathbbm{k}[x_{1},...,x_{d}]$ considered here comes from viewing $\mathbbm{k}[x_{1},...,x_{d}]$ as the coordinate ring of the affine algebraic group $(\mathbbm{k}^{d},+)$) and $B$ has the Chevalley property by Proposition \ref{prop-tensorChevisChev}. 
Note that its identity fiber algebra is $\mathbbm{k}G$, by Proposition \ref{prop-semiprimeChev} and \cite[Theorem 0.2 (1)]{MR1938745}, $B$ is semiprime and Artin-Schelter regular. Thus, for any positive integer $d$, there are Artin-Schelter regular Hopf algebras of GK-dimension $d$ that are not commutative, admit large central Hopf subalgebras, and have the Chevalley property.

%

\subsection{Rigidity of the lowest discriminant varieties}
Theorem \ref{thm-lowestsetsubgrp} below reflects the \textit{rigid} nature of the lowest discriminant subvarieties of Cayley-Hamilton Hopf algebras and helps \textit{predict} and \textit{determine} the lowest discriminant subvarieties in certain examples of low GK-dimension. Hence, it may be regarded as a rigidity theorem for the lowest discriminant varieties.
\begin{theorem} \label{thm-lowestsetsubgrp}
Let $(H,C,\text{tr})$ be an affine Cayley-Hamilton Hopf algebra such that the identity fiber algebra $H/\mathfrak{m}_{\ovep}H$ has the Chevalley property. Then the lowset discriminant subvariety $\mathcal{V}( D_{\ell}(H/C;\text{tr}) )$ is a subgroup of the affine algebraic group $\operatorname{maxSpec}C$. 
\end{theorem}
\begin{proof}
Since $\text{Sd}(\mathfrak{m})=\text{Sd}(\mathfrak{m}^{-1})$ for all $\mathfrak{m}\in \operatorname{maxSpec}C$, it follows from 
\eqref{eq-thm-BY-disrep} that the subvariety $\mathcal{V}( D_{\ell}(H/C;\text{tr}) )$ is closed under taking inverses.
It remains to show that
\begin{equation}\label{eq-lowesvar-cm1}
\mathfrak{m}\ast\mathfrak{n}\in \mathcal{V}(D_{\ell}(H/C;\text{tr}))\text{ for all\ }\mathfrak{m},\mathfrak{n}\in  \mathcal{V}(D_{\ell}(H/C;\text{tr})).
\end{equation}
For any irreducible $H/\mathfrak{m}H$-module $V$ and irreducible $H/\mathfrak{n}H$-module $W$, by Theorem \ref{thm-tensred-insec2}, both $V$ and $W$ are tensor-reducible as $H$-modules. It follows that $V\otimes W$ is tensor-reducible as an $H$-module. Note that $\mathfrak{m}\ast \mathfrak{n}$ is a maximal ideal of $C$ and $(\mathfrak{m}\ast\mathfrak{n})(V\otimes W)=0$.
 Then $V\otimes W$ can be regarded as an $H/(\mathfrak{m}\ast\mathfrak{n})H$-module, and $\text{Ann}_{C}(V\otimes W)=\mathfrak{m}\ast\mathfrak{n}$. By Theorem \ref{thm-tensred-insec2}, $D_{\ell}(H/C;\text{tr})  \subseteq \mathfrak{m}\ast\mathfrak{n}  $. This proves \eqref{eq-lowesvar-cm1}.
\end{proof}

Let $(H,C,\text{tr})$ be an affine Cayley-Hamilton Hopf algebra such that the identity fiber algebra $H/\mathfrak{m}_{\ovep}H$ has the Chevalley property. In Theorem \ref{thm-Cheviff-ap}, we see that whether $H$ has the Chevalley property is closely related to the discriminant ideals, especially the lowest one.  By Theorem \ref{thm-lowestsetsubgrp}, the lowest discriminant subvariety $\mathcal{V}_{\ell}=\mathcal{V}(D_{\ell}(H/C;\text{tr}))$ is a closed subgroup of the affine algebraic group $\operatorname{maxSpec}C$. In particular, the coordinate ring of $\mathcal{V}_{\ell}$,
$$C/\sqrt{D_{\ell}(H/C;\text{tr})},$$ is a Hopf quotient of $C$. Hence, the Hopf ideal $\sqrt{D_{\ell}(H/C;\text{tr})}$ of $C$ gives rise to a (not necessary finite-dimensional) Hopf quotient of $H$:
$$\frac{H}{\sqrt{D_{\ell}(H/C;\text{tr})}H}.$$
It is easy to see that $C/\sqrt{D_{\ell}(H/C;\text{tr})}$ is a central Hopf subalgebra of $H/\sqrt{D_{\ell}(H/C;\text{tr})}H$. In fact, for any ideal $\mathfrak{a}$ of $C$, $\mathfrak{a}H\cap C=\mathfrak{a}$ since $\text{tr}:H\to C$ is $C$-linear and $\text{tr}(c)=nc$ for all $c\in C$. It is worth noting that the above Hopf quotient satisfies the following:
\begin{corollary}
Let $(H,C,\text{tr})$ be an affine Cayley-Hamilton Hopf algebra of degree $n$ such that the identity fiber algebra $H/\mathfrak{m}_{\ovep}H$ has the Chevalley property. Let $\mathfrak{d}_{\ell}= \sqrt{D_{\ell}(H/C;\text{tr})}\subseteq C$ be the radical of the lowest discriminant ideal $ D_{\ell}(H/C;\text{tr})$. Let $\text{tr}_{\mathfrak{d}_{\ell}}:H/\mathfrak{d}_{\ell}H\to C/\mathfrak{d}_{\ell}$ denote the canonical map induced by $\text{tr}$. Then the following hold.\label{cor-affCHquodisHq-cheviffid}
\begin{itemize}
\item[(a)] The algebra with trace $(H/\mathfrak{d}_{\ell}H, C/\mathfrak{d}_{\ell}, \text{tr}_{\mathfrak{d}_{\ell}})$ is an affine Cayley-Hamilton Hopf algebra with the Chevalley property, and its identity fiber algebra is $H/\mathfrak{m}_{\ovep}H$.
\item[(b)]For any Hopf ideal $\mathfrak{a}$ of $C$, the Hopf quotient $H/\mathfrak{a}H$ has the Chevalley property if and only if $\mathfrak{a}\supseteq \mathfrak{d}_{\ell}$.
\end{itemize}
\end{corollary}
\begin{proof}
(a) By Theorem \ref{thm-lowestsetsubgrp} and the preceding discussion, $\mathfrak{d}_{\ell}$ is a Hopf ideal of $C$ and $C/\mathfrak{d}_{\ell}$ can be regarded as a central Hopf subalgebra of $H/\mathfrak{d}_{\ell}H$. Moreover, the canonical projectrion $\pi_{\mathfrak{d}_{\ell}}: H\to H/\mathfrak{d}_{\ell}H$ is trace-preserving, that is, one has the following commutative diagram:
$$\begin{tikzcd}
H \arrow[rr, "\pi_{\mathfrak{d}_{\ell}}"] \arrow[d, "\text{tr}"'] &  & H/\mathfrak{d}_{\ell}H \arrow[d, "\text{tr}_{\mathfrak{d}_{\ell}}"] \\
C \arrow[rr]                                                      &  & C/\mathfrak{d}_{\ell}                                              
\end{tikzcd}$$
Now it is clear that $(H/\mathfrak{d}_{\ell}H, C/\mathfrak{d}_{\ell}, \text{tr}_{\mathfrak{d}_{\ell}})$ is an affine Cayley-Hamilton Hopf algebra of degree $n$. Since $\mathfrak{m}_{\ovep}\supseteq \mathfrak{d}_{\ell}$, the identity fiber of $(H/\mathfrak{d}_{\ell}H, C/\mathfrak{d}_{\ell})$ is $(H/\mathfrak{d}_{\ell}H)/(\mathfrak{m}_{\ovep}H/\mathfrak{d}_{\ell}H)\cong H/\mathfrak{m}_{\ovep}H$. For any irreducible $H/\mathfrak{d}_{\ell}H$-module $V$, $V$ is also irreducible as an $H$-module and $\mathfrak{d}_{\ell}\subseteq \text{Ann}_{H}V$. Thus, by Theorem \ref{thm-tensred-insec2} (a), $V$ is a tensor-reducible $H$-module. In particular, the tensor product of any two irreducible modules of $H/\mathfrak{d}_{\ell}H$ is completely reducible.

(b) The sufficiency part is clear, now we prove necessity. As in (a), $(H/\mathfrak{a}H,C/\mathfrak{a},\text{tr}_{\mathfrak{a}})$ is an affine Cayley-Hamilton Hopf algebra of degree $n$, here $\text{tr}_{\mathfrak{a}}:H/\mathfrak{a}H\to C/\mathfrak{a}$ is the induced canonical trace map. The identity fiber algebra of $(H/\mathfrak{a}H,C/\mathfrak{a})$ is $H/\mathfrak{m}_{\ovep}H$. Since $H/\mathfrak{a}H$ has the Chevalley property, by Theorem \ref{thm-Cheviff-ap} and \eqref{eq-thm-BY-disrep}, any maximal ideal of $C/\mathfrak{a}$, say $\mathfrak{m}/\mathfrak{a}$, satisfies $\text{Sd}_{H}(\mathfrak{m})=\text{Sd}_{H/\mathfrak{a}H}(\mathfrak{m}/\mathfrak{a})=\text{Sd}_{H/\mathfrak{a}H}(\mathfrak{m}_{\ovep}/\mathfrak{a})=\text{Sd}_{H}(\mathfrak{m}_{\ovep})$. It follows that $\mathfrak{m}\supseteq \mathfrak{d}_{\ell}$. Since $\text{char}\mathbbm{k}=0$, $C/\mathfrak{a}$ is reduced. It follows that, since $C$ is a Jacobson ring, $\mathfrak{a}\supseteq  \mathfrak{d}_{\ell}$.
\end{proof}

Corollary \ref{cor-affCHquodisHq-cheviffid} enables us to construct tensor categories with the Chevalley property from some big quantum groups at roots of unity. 

\begin{example}\label{eg-chevtensorcat-fromquanBorel}
Retain hypotheses and notation of \S \ref{subsec-Uqb}. Then the De Concini-Kac-Procesi quantized Borel subalgebra $\mathcal{U}_{\epsilon}^{\geq 0}(\mathfrak{g})$, endowed with its canonical central Hopf subalgebra $\mathcal{C}_{\epsilon}^{\geq 0}(\mathfrak{g})$ and the regular trace, gives rise to an affine Cayley-Hamilton Hopf algebra $$(\mathcal{U}_{\epsilon}^{\geq 0}(\mathfrak{g}),\mathcal{C}_{\epsilon}^{\geq 0}(\mathfrak{g}),\text{tr}_{\text{reg}}).$$
Moreover, the identity fiber is basic. By \cite[Theorem 5.5 (b)]{mi2025lowest}, the lowest discriminant subvariety of $(\mathcal{U}_{\epsilon}^{\geq 0}(\mathfrak{g}),\mathcal{C}_{\epsilon}^{\geq 0}(\mathfrak{g}),\text{tr}_{\text{reg}})$ is of level $l^{n}+1$ and
$$\mathcal{V}(D_{l^{n}+1}(\mathcal{U}_{\epsilon}^{\geq 0}(\mathfrak{g})/\mathcal{C}_{\epsilon}^{\geq 0}(\mathfrak{g});\text{tr}_{\text{reg}}))=\{(K_{1}^{l}-\alpha_{1},...,K_{n}^{l}-\alpha_{n},E_{\beta_{1}}^{l},...,E_{\beta_{N}}^{l})\mid \alpha_{1},...,\alpha_{n}\in \mathbbm{k}^{\times}\}.$$
It follows that
$$\mathfrak{d}_{l^{n}+1}=\sqrt{D_{l^{n}+1}(\mathcal{U}_{\epsilon}^{\geq 0}(\mathfrak{g})/\mathcal{C}_{\epsilon}^{\geq 0}(\mathfrak{g});\text{tr}_{\text{reg}})}=(E_{\beta_{1}}^{l},...,E_{\beta_{N}}^{l})\subseteq \mathcal{C}_{\epsilon}^{\geq 0}(\mathfrak{g}).  $$
 Clearly, $\mathcal{C}_{\epsilon}^{\geq 0}(\mathfrak{g})/\mathfrak{d}_{l^{n}+1}\cong \mathbb{C}[K_{1}^{\pm l},...,K_{n}^{\pm l} ]$. Hence $\text{GK dimension }\mathcal{U}_{\epsilon}^{\geq 0}(\mathfrak{g})/\mathfrak{d}_{l^{n}+1}\mathcal{U}_{\epsilon}^{\geq 0}(\mathfrak{g})=n$ and $$(\mathcal{U}_{\epsilon}^{\geq 0}(\mathfrak{g})/\mathfrak{d}_{l^{n}+1}\mathcal{U}_{\epsilon}^{\geq 0}(\mathfrak{g}),C/\mathfrak{d}_{l^{n}+1},\text{tr}_{\text{reg}} )$$ is an infinite-dimensional Cayley-Hamilton Hopf algebra with the Chevalley property. Namely, $(\mathcal{U}_{\epsilon}^{\geq 0}(\mathfrak{g})/\mathfrak{d}_{l^{n}+1}\mathcal{U}_{\epsilon}^{\geq 0}(\mathfrak{g}))\text{-mod}$ is a tensor category with the Chevalley property.
\end{example}

The following example demonstrates that, in the setting of Theorem \ref{thm-lowestsetsubgrp}, a non-lowest discriminant subvariety is not, in general, a subgroup of $\operatorname{maxSpec}C$.
\begin{example}
[\cite{MR1898492,MR1296515}]Let $r\geq 3$ be an odd positive integer and let $\epsilon\in \mathbb{C}$ be a primitive $r$-th root of unity. Recall that the quantized function algebra $\mathcal{O}_{\epsilon}(\text{SL}_{2}(\mathbb{C}))$ is the Hopf algebra generated by $a,b,c,d$, subject to\label{eg-OepSL2-notsubgrp} relations
$$ab=\epsilon ba,\  ac=\epsilon ca,\  bc=cb,\  bd=\epsilon db,\  cd=\epsilon dc,$$
$$ad-da=(\epsilon-\epsilon^{-1})bc,\  ad-\epsilon bc=1.$$
The coalgebra structure of $\mathcal{O}_{\epsilon}(\text{SL}_{2}(\mathbb{C}))$ is defined by
$$\Delta(a)=a\otimes a+b\otimes c,\  \Delta(b)=a\otimes b+b\otimes d,\ \varepsilon(a)=1,\ \varepsilon(b)=0,$$
$$\Delta(c)=d\otimes c+c\otimes a,\  \Delta(d)=d\otimes d+c\otimes b,\ \varepsilon(c)=0,\  \varepsilon(d)=1.$$
Let $\mathcal{C}_{\epsilon}(\text{SL}_{2}(\mathbb{C}))$ be the subalgebra of $\mathcal{O}_{\epsilon}(\text{SL}_{2}(\mathbb{C}))$ generated by $a^{r},b^{r},c^{r},d^{r}$. It follows from \cite[Proposition. III. 3.1]{MR1898492} and Example \ref{eg-HStrace} that $(\mathcal{O}_{\epsilon}(\text{SL}_{2}(\mathbb{C})),\mathcal{C}_{\epsilon}(\text{SL}_{2}(\mathbb{C}),\text{tr}_{HS})$ is an affine Cayley-Hamilton Hopf algebra of degree $r^{3}$ and $\mathcal{C}_{\epsilon}(\text{SL}_{2}(\mathbb{C}))\cong \mathcal{O}(\text{SL}_{2}(\mathbb{C}))$ as Hopf algebras. Here $\text{tr}_{HS}:\mathcal{O}_{\epsilon}(\text{SL}_{2}(\mathbb{C}))\to \mathcal{C}_{\epsilon}(\text{SL}_{2}(\mathbb{C}))$ is the Hattori-Stallings trace map. Clearly,
\begin{align*}
\operatorname{maxSpec}\mathcal{C}_{\epsilon}(\text{SL}_{2}(\mathbb{C}))=&\{(a^{r}-\lambda,b^{r}-\mu,c^{r}-\xi,d^{r}-\eta)\mid  \lambda,\mu,\xi,\eta\in \mathbb{C},\lambda\eta-\mu\xi=1 \}\\
\cong&\left\{\begin{pmatrix}
\lambda &\mu\\
\xi& \eta
\end{pmatrix}\in \text{M}_{2}(\mathbb{C}) \mid \lambda\eta-\mu\xi=1
\right\}\\
=& \text{SL}_{2}(\mathbb{C})
\end{align*}
as algebraic groups. We use the above isomorphism to identify $\operatorname{maxSpec}\mathcal{C}_{\epsilon}(\text{SL}_{2}(\mathbb{C}))$ with $\text{SL}_{2}(\mathbb{C})$. By \cite[III. \S 3.2]{MR1898492} and \eqref{eq-thm-BY-disrep}, $\mathcal{O}_{\epsilon}(\text{SL}_{2}(\mathbb{C}))/\mathfrak{m}_{\ovep}\mathcal{O}_{\epsilon}(\text{SL}_{2}(\mathbb{C}))$ is basic, the lowest discriminant ideal is of level $r+1$, 
$$\mathcal{V}(D_{r+1}(\mathcal{O}_{\epsilon}(\text{SL}_{2}(\mathbb{C}))/\mathcal{C}_{\epsilon}(\text{SL}_{2}(\mathbb{C}));\text{tr}_{HS}))=\left\{\begin{pmatrix}
\lambda &0\\
0& \lambda^{-1}
\end{pmatrix}\mid \lambda\in \mathbb{C}^{\times}\right\},$$
$$\mathcal{V}(D_{r^{2}+1}(\mathcal{O}_{\epsilon}(\text{SL}_{2}(\mathbb{C}))/\mathcal{C}_{\epsilon}(\text{SL}_{2}(\mathbb{C}));\text{tr}_{HS}))=\left\{\begin{pmatrix}
\lambda &\mu\\
0& \lambda^{-1}
\end{pmatrix},\begin{pmatrix}
\lambda &0\\
\mu& \lambda^{-1}
\end{pmatrix} \mid \lambda\in \mathbb{C}^{\times},\mu\in \mathbb{C}\right\},\ \text{and}$$
$$\mathcal{V}(D_{r^{3}+1}(\mathcal{O}_{\epsilon}(\text{SL}_{2}(\mathbb{C}))/\mathcal{C}_{\epsilon}(\text{SL}_{2}(\mathbb{C}));\text{tr}_{HS}))=\operatorname{maxSpec}\mathcal{C}_{\epsilon}(\text{SL}_{2}(\mathbb{C})).$$
Hence, one has the following ascending chain of the discriminant subvarieties:
\begin{align*}
\varnothing=\cdots=&\mathcal{V}(D_{r}(\mathcal{O}_{\epsilon}(\text{SL}_{2}(\mathbb{C}))/\mathcal{C}_{\epsilon}(\text{SL}_{2}(\mathbb{C}));\text{tr}_{HS}))\subsetneq \mathcal{V}(D_{r+1}(\mathcal{O}_{\epsilon}(\text{SL}_{2}(\mathbb{C}))/\mathcal{C}_{\epsilon}(\text{SL}_{2}(\mathbb{C}));\text{tr}_{HS}))\\
=\cdots=& \mathcal{V}(D_{r^{2}}(\mathcal{O}_{\epsilon}(\text{SL}_{2}(\mathbb{C}))/\mathcal{C}_{\epsilon}(\text{SL}_{2}(\mathbb{C}));\text{tr}_{HS}))\subsetneq  \mathcal{V}(D_{r^{2}+1}(\mathcal{O}_{\epsilon}(\text{SL}_{2}(\mathbb{C}))/\mathcal{C}_{\epsilon}(\text{SL}_{2}(\mathbb{C}));\text{tr}_{HS}))\\
=\cdots=& \mathcal{V}(D_{r^{3}}(\mathcal{O}_{\epsilon}(\text{SL}_{2}(\mathbb{C}))/\mathcal{C}_{\epsilon}(\text{SL}_{2}(\mathbb{C}));\text{tr}_{HS}))\subsetneq \mathcal{V}(D_{r^{3}+1}(\mathcal{O}_{\epsilon}(\text{SL}_{2}(\mathbb{C}))/\mathcal{C}_{\epsilon}(\text{SL}_{2}(\mathbb{C}));\text{tr}_{HS}))\\
=\cdots=&\operatorname{maxSpec}\mathcal{C}_{\epsilon}(\text{SL}_{2}(\mathbb{C})).
\end{align*}

Obviously, $\mathcal{V}(D_{r^{2}+1}(\mathcal{O}_{\epsilon}(\text{SL}_{2}(\mathbb{C}))/\mathcal{C}_{\epsilon}(\text{SL}_{2}(\mathbb{C}));\text{tr}_{HS}))$ is not a subgroup of $\operatorname{maxSpec}\mathcal{C}_{\epsilon}(\text{SL}_{2}(\mathbb{C}))$.
\end{example}
 Theorem \ref{thm-lowestsetsubgrp} is effective in determining the lowest discriminant subvarieties for some examples of low GK-dimension (i.e., at most $2$).
\begin{example}
[\cite{MR2732991}]\label{eg-GZalgA1} Let $l\geq 2$ be a positive integer, $\xi $ be a primitive $l$-th root of unity and $n$ be a positive integer coprime to $l$. Let $H=\mathcal{A}(n,\xi)$ be the $\mathbbm{k}$-algebra generated by $x^{\pm 1},y$ subject to the relations $xx^{-1}=x^{-1}x=1,xy=\xi yx$. By \cite[Construction 1.1]{MR2732991}, there is a unique Hopf algebra structure on $H$ such that
$$\Delta(x)=x\otimes x,\  \varepsilon(x)=1,$$
$$\Delta(y)=y\otimes 1+x^{n}\otimes y,\  \varepsilon(y)=0.$$
Set $\sigma:\mathbbm{k}[x,x^{-1}]\to \mathbbm{k}[x,x^{-1}],x\mapsto \xi^{-1}x$. Then it is clear that 
$$H\cong \mathbbm{k}[x,x^{-1}][y;\sigma].$$
Hence $\{x^{i}y^{j}\mid i\in \mathbb{Z},j\in \mathbb{N}\}$ is a $\mathbbm{k}$-basis of $H$. Moreover, $H$ is an affine prime Artin-Schelter regular Hopf domain of GK-dimension two.

Note that $\xi^{n}$ is a primitive $l$-th root of unity. Using the quantum binomial theorem, one easily verifies that $\Delta(y^{l})=y^{l}\otimes 1+x^{nl}\otimes y^{l}$ and thus $C=\mathbbm{k}[x^{\pm l},y^{l}]$ is a central Hopf subalgebra of $H$. Hence $H$ is a free $C$-module of rank $l^{2}$. It follows from Example \ref{eg-reutr1} that $(H,C,\text{tr}_{\text{reg}})$ is an affine Cayley-Hamilton Hopf algebra of degree $l^{2}$. 

Clearly, the identity fiber algebra $H/\mathfrak{m}_{\ovep}H=\mathbbm{k}\langle x,y\rangle/(x^{l}-1,xy-\xi yx, y^{l})$ is precisely the $l^{2}$-dimensional Taft algebra, which is basic. By \cite[Theorem 4.2]{mi2025lowest} or \cite[Theorem 4.1]{HMQWChev2025}, the lowest discriminant ideal of $(H,C,\text{tr}_{\text{reg}})$ is of level $\ell\coloneqq l+1$. Denote the additive group of $\mathbbm{k}$ by $\mathbbm{k}^{+}$, and consider the action of the multiplicative group $\mathbbm{k}^{\times}$ on $\mathbbm{k}^{+}$ defined by $\varphi:\mathbbm{k}^{\times }\to \text{Aut}(\mathbbm{k}^{+}),  \beta\mapsto (\alpha\mapsto \beta^{n}\alpha)$. Then one has the semidirect product $\mathbbm{k}^{+}\rtimes_{\varphi}\mathbbm{k}^{\times}$. The multiplication of $\mathbbm{k}^{+}\rtimes_{\varphi}\mathbbm{k}^{\times}$ is given by
$$(\alpha_{1},\beta_{1})(\alpha_{2},\beta_{2})=(\alpha_{1}+\beta_{1}^{n}\alpha_{2},\beta_{1}\beta_{2}),\text{ for all\ }\alpha_{1},\alpha_{2}\in \mathbbm{k},\beta_{1},\beta_{2}\in \mathbbm{k}^{\times}.$$

Then the algebraic group $\operatorname{maxSpec}C$ is isomorphic to the semidirect product $\mathbbm{k}^{+}\rtimes_{\varphi}\mathbbm{k}^{\times}$. Indeed, the map
$$\mathbbm{k}^{+}\rtimes_{\varphi}\mathbbm{k}^{\times}\to \operatorname{maxSpec}C, (\alpha,\beta)\mapsto \mathfrak{m}_{\alpha,\beta}\coloneqq (y^{l}-\alpha,x^{l}-\beta)$$
gives an isomorphism from $\mathbbm{k}^{+}\rtimes_{\varphi}\mathbbm{k}^{\times}$ to $ \operatorname{maxSpec}C$. Note that $\mathfrak{m}_{\ovep}=\mathfrak{m}_{0,1}$. Below we exploit the algebraic group structure on the semidirect product $\mathbbm{k}^{+}\rtimes_{\varphi}\mathbbm{k}^{\times}$ to deduce that
\begin{equation}\label{eq-GZ2exlow}
\mathcal{V}(D_{\ell}(H/C;\text{tr}))=\{\mathfrak{m}_{0,\beta}\in \operatorname{maxSpec}C\mid  \beta\in \mathbbm{k}^{\times}\}.
\end{equation}
For any $\beta\in \mathbbm{k}^{\times}$, it is clear that the fiber algebra $H/\mathfrak{m}_{0,\beta}H$ has a character $\chi_{\beta}: H/\mathfrak{m}_{0,\beta}H\to\mathbbm{k}$ satisfying $\chi_{\beta}(\overline{x})=\beta^{1/l}$ and $\chi_{\beta}(\overline{y})=0$. By \cite[Theorem 4.3 (a)(b)]{mi2025lowest}, $H/\mathfrak{m}_{0,\beta}H\cong H/\mathfrak{m}_{\ovep}H $ and $\mathfrak{m}_{0,\beta}\in \mathcal{V}(D_{\ell}(H/C;\text{tr}))$. By \cite[Corollary 4.10]{HMQWChev2025}, $H$ does not have the Chevalley property and $\mathcal{V}(D_{\ell}(H/C;\text{tr}))$ is a proper closed subgroup of $\operatorname{maxSpec}C$. 

We now show that $\mathscr{T}\coloneqq \{(0,\beta)\mid \beta\in \mathbbm{k}^{\times}\}$ is a maximal subgroup of $\mathbbm{k}^{+}\rtimes_{\varphi}\mathbbm{k}^{\times}$, thereby establishing \eqref{eq-GZ2exlow}. Let $\mathscr{G}$ be a subgroup of the semidirect product $\mathbbm{k}^{+}\rtimes_{\varphi}\mathbbm{k}^{\times}$ that contains $\mathscr{T}$. Assume that $\mathscr{G}$ contains an element $(\alpha_{0},\beta_{0})$ with $\alpha_{0}\neq 0$. For any $\gamma\in \mathbbm{k}$, consider $(0,(\gamma\alpha_{0}^{-1})^{1/n})\in \mathscr{T}$, one has $(\gamma,\beta_{0})=(0,(\gamma\alpha_{0}^{-1})^{1/n})(\alpha_{0},\beta_{0})\in \mathscr{G}$. Hence $(\gamma,\beta)\in \mathscr{G}$ for all $\gamma\in \mathbbm{k}$ and $\beta\in \mathbbm{k}^{\times}$. That is, $\mathscr{G}=\mathbbm{k}^{+}\rtimes_{\varphi}\mathbbm{k}^{\times}$. This completes the proof of \eqref{eq-GZ2exlow}.

By \eqref{eq-GZ2exlow}, one has $\mathfrak{d}_{\ell}=\sqrt{D_{\ell}(H/C;\text{tr})}=(y^{l})$. Hence $H/y^{l}H$ has the Chevalley property by Corollary \ref{cor-affCHquodisHq-cheviffid}. It is not hard to see that the family of Hopf algebras with the Chevalley property considered in Example \ref{eg-aCheegwith1dimz-u2} can be obtained in this way.

Clearly, $\mathcal{V}(D_{\ell}(H/C;\text{tr}))$ is not a normal subgroup of $\operatorname{maxSpec}C$. 
\end{example}
\begin{remark}
Example \ref{eg-GZalgA1} shows that, in the setting of Theorem \ref{thm-lowestsetsubgrp}, the lowest discriminant subvariety of $(H,C,\text{tr})$ is not necessarily a normal subgroup of $\operatorname{maxSpec}C$.
\end{remark}

The corollary below shows that Theorem \ref{thm-lowestsetsubgrp} imposes a strong restriction on the possible shapes of the lowest discriminant subvariety for prime examples of GK-dimension one.

\begin{corollary}\label{cor-GK1prim-low}
Let $(H,C,\text{tr})$ be an affine Cayley-Hamilton Hopf algebra such that the identity fiber algebra $H/\mathfrak{m}_{\ovep}H$ has the Chevalley property. Assume that $H$ is prime and is of GK-dimension one. Then either $H$ is commutative or the lowest discriminant subvariety of $(H,C,\text{tr})$ is a finite cyclic subgroup of the algebraic group $\operatorname{maxSpec}C$.
\end{corollary}
\begin{proof}
If $H$ has the Chevalley property, then $H$ is commutative by \cite[Corollary 4.10]{HMQWChev2025}. If $H$ does not have the Chevalley property, then  the lowest discriminant subvariety of $(H,C,\text{tr})$ is a proper closed subvariety of $\operatorname{maxSpec}C$ by \eqref{eq-thm-BY-disrep} and Theorem \ref{thm-Cheviff-ap}. By \cite[Proposition 8.2.9 (ii)]{MR1811901}, $\operatorname{maxSpec}C$ is a connected affine algebraic group of dimension one. Hence \cite[Theorem 20.5]{MR396773} implies that as an algebraic group, $\operatorname{maxSpec}C$ is isomorphic to either the additive group $\mathbbm{k}^{+}$ or the multiplicative group $\mathbbm{k}^{\times}$. Thus it suffices to show that every proper subgroup of $\mathbbm{k}^{+}$ and $\mathbbm{k}^{\times}$ is finite cyclic. Note that the only finite subgroup of the additive group $\mathbbm{k}^{+}$ is the trivial one, and any finite subgroup of the multiplicative group of an infinite field is cyclic. So the result holds.
\end{proof}
As a simple example, consider the infinite Taft algebra $\mathcal{H}(n,t,\xi)$ \cite{MR2320655}. As a $\mathbbm{k}$-algebra, $\mathcal{H}(n,t,\xi)$ is generated by $x$ and $g$ with relations $g^{n}=1$ and $xg=\xi gx$, where $n\geq 2,1\leq t\leq n$ are two coprime integers, and $\xi\in \mathbbm{k}^{\times}$ is a primitive $n$-th root of unity. There is a unique Hopf algebra structure on $\mathcal{H}(n,t,\xi)$ such that $\Delta(g)=g\otimes g,  \Delta(x)=x\otimes g^{t}+1\otimes x,   \varepsilon(g)=1,$ and $ \varepsilon(x)=0.$
Then $\mathcal{C}(n,t,\xi)=\mathbbm{k}[x^{n}]$ is a central Hopf subalgebra of $\mathcal{H}(n,t,\xi)$ and $\mathcal{H}(n,t,\xi)$ is a free $\mathcal{C}(n,t,\xi)$-module of rank $n^{2}$. Thus $(\mathcal{H}(n,t,\xi),\mathcal{C}(n,t,\xi),\text{tr}_{\text{reg}})$ is an affine Cayley-Hamilton Hopf algebra of degree $n^{2}$ and its identity fiber algebra is the $n^{2}$-dimensional Taft algebra. Moreover, $\mathcal{H}(n,t,\xi)$ is a prime  Artin-Schelter regular Hopf algebra of GK dimension one. In \cite{HMQWChev2025}, by computing the corresponding square dimension function on $\operatorname{maxSpec}C$, the authors show that the lowest discriminant subvariety $\mathcal{V}(D_{\ell}(\mathcal{H}(n,t,\xi)/\mathcal{C}(n,t,\xi);\text{tr}_{\text{reg}}))$ is a singleton. This can also be obtained immediately from Corollary \ref{cor-GK1prim-low}. Since $H$ is not commutative and $\operatorname{maxSpec}C\cong (\mathbbm{k},+)$ as algebraic groups, the only finite subgroup of $\operatorname{maxSpec}C$ is the trivial group. Hence the lowest discriminant subvariety $\mathcal{V}(D_{\ell}(\mathcal{H}(n,t,\xi)/\mathcal{C}(n,t,\xi);\text{tr}_{\text{reg}}))$ is a singleton.

The following example shows that, every finite cyclic group arises as the lowest discriminant subvariety of some affine prime Cayley-Hamilton Hopf algebra $(H,C,\text{tr})$ of GK-dimension one whose identity fiber algebra has the Chevalley property.

\begin{example}
[\cite{MR2661247}] \label{eg-geLiualg}Let $n$ and $w$ be positive integers with $n\geq 2$, and let $\xi$ be a primitive $n$-th root of unity.  Consider the generalized Liu algebra $H=\mathcal{B}(n,w,\xi)$. As a $\mathbbm{k}$-algebra, $H$ is generated by $x^{\pm 1},g,y$ with relations $yg=\xi gy,y^{n}=1-x^{w}=1-g^{n},xy=yx,xx^{-1}=x^{-1}x=1$ and $xg=gx$. By \cite[\S 3.4]{MR2661247}, there is a unique Hopf algebra structure on $H$ such that $$\Delta(x)=x\otimes x,\ \Delta(g)=g\otimes g,\ \varepsilon(x)=\varepsilon(g)=1,$$
$$\Delta(y)=y\otimes g+1\otimes y,\ \varepsilon(y)=0.$$
By \cite[Theorem 3.4 (b)(c)(e)]{MR2661247} and \cite[Theorem 0.2 (1)]{MR1938745}, $H$ is a prime Artin-Schelter regular Hopf algebra of GK-dimension one with center $C=\mathbbm{k}[x^{\pm 1}]$. Then $H$ is a free $C$-module with basis $\{y^{i}g^{j}\mid 0\leq i,j\leq n-1\}$. Hence $(H,C,\text{tr}_{\text{reg}})$ is an affine Cayley-Hamilton Hopf algebra of degree $n^{2}$. Now we show that the corresponding lowest discriminant subvariety is the cyclic subgroup of order $w$ of the algebraic group $\operatorname{maxSpec}C\cong \mathbbm{k}^{\times}$.

For any $\alpha\in \mathbbm{k}^{\times}$, set $\mathfrak{m}_{\alpha}=(x-\alpha)\in \operatorname{maxSpec}C$. Clearly, $\mathfrak{m}_{\ovep}=\mathfrak{m}_{1}$ and the identity fiber algebra $H/\mathfrak{m}_{\ovep}H$ is the $n^{2}$-dimensional Taft algebra. Indeed, it is straightforward to show that, for any $\alpha\in \mathbbm{k}^{\times}$ such that $\alpha^{w}=1$, $H/\mathfrak{m}_{\alpha}H$ is likewise the $n^{2}$-dimensional Taft algebra. If $\alpha\in \mathbbm{k}^{\times}$ satisfies $\alpha^{w}\neq 1$, then there is a unique algebra homomorphism $\rho_{\alpha}:H\to\text{M}_{n}(\mathbbm{k})$ such that
$$\rho_{\alpha}(x)=\alpha I_{n}, \ \rho_{\alpha}(y)=(1-\alpha^{w})^{1/n}\begin{pmatrix}
 &1 & & & \\
 & &1& &\\
& &  &\ddots &\\
& & & &1 \\
1& & & &
\end{pmatrix},\ \text{and}$$
$$\rho_{\alpha}(g)=\alpha^{w/n}\begin{pmatrix}
1& & & & \\
 &\xi & & &\\
& & \xi^{2} & &\\
& & &\ddots &\\
& & & &\xi^{n-1} 
\end{pmatrix}.$$
Clearly, $\rho_{\alpha}$ induces a surjective algebra homomorphism $\overline{\rho_{\alpha}}:H/\mathfrak{m}_{\alpha}H\to \text{M}_{n}(\mathbbm{k})$. By \cite[Proposition 2.1 (1)]{HMQWChev2025}, $\dim_{\mathbbm{k}}H/\mathfrak{m}_{\alpha}H=\dim_{\mathbbm{k}}H/\mathfrak{m}_{\ovep}H=n^{2}$. Hence we obtain that $H/\mathfrak{m}_{\alpha}H\cong \text{M}_{n}(\mathbbm{k})$ as $\mathbbm{k}$-algebras. By \eqref{eq-thm-BY-disrep}, one has
$$
\mathcal{V}(D_{k}(H/C;\text{tr}_{\text{reg}}))=\begin{cases}
\varnothing,& 1\leq k\leq n,\\
\{\mathfrak{m}_{\alpha}\mid \alpha^{w}=1\},& n+1\leq k\leq n^{2},\\
\operatorname{maxSpec}C,& k \geq  n^{2}+1.
\end{cases}
$$
Hence $\mathcal{V}(D_{n+1}(H/C;\text{tr}_{\text{reg}}))$ is the lowest discriminant subvariety of $(H,C,\text{tr}_{\text{reg}})$ and it is a cyclic group of order $w$. In this case, $\mathfrak{d}_{n+1}=\sqrt{D_{n+1}(H/C;\text{tr}_{\text{reg}})}=(x^{w}-1)\subseteq C$. It follows that $\dim _{\mathbbm{k}}C/\mathfrak{d}_{n+1}=w$ and hence $H/\mathfrak{d}_{n+1}H$ is an $n^{2}w$-dimensional Hopf algebra with the Chevalley property.
\end{example}

\section*{Acknowledgments}
The authors thank Zhongkai Mi and Milen Yakimov for helpful discussions and comments.


\end{document}